\documentclass[11pt,reqno]{amsart}

\usepackage[utf8]{inputenc} 
\usepackage[english]{babel} 
\usepackage{mathtools,amssymb,amsthm,latexsym,amsmath,amstext} 
\usepackage{xcolor} 
\usepackage{titletoc} 
\usepackage{enumitem} 
\usepackage{mathrsfs} 
\usepackage{commath} 
\usepackage{listings}
\usepackage{caption}
\usepackage{hhline}
\usepackage{tikz}
\usepackage{tikz-cd}
\usepackage{bbold} 
\usepackage{todonotes}
\usepackage{varioref} 
\usepackage{nameref}
\usepackage{amsthm,xpatch}
\usepackage{lipsum}
\usepackage{eucal} 
\usepackage{lipsum}
\usepackage{mathtools}
\usepackage{dsfont}
\usepackage{hyperref}
\usepackage{braket}
\usepackage{amsmath,amsfonts,amstext,amssymb,amsbsy,amsopn,amsthm,eucal}
\usepackage{tabularx}
\usepackage{adjustbox}
\usepackage[labelsep=none]{caption}
\usepackage{textcomp}
\usepackage{placeins}

\newcommand \M{\mathfrak{M}}
\newcommand \Image{\mathrm{Im}}
\newcommand \Aff{\mathrm{Aff}}

\newcommand \Ker{\mathrm{Ker}}
\newcommand \X{\mathcal{X}}

\newcommand \N{\mathbb N}
\newcommand \Z{\mathbb Z}

\newcommand \R{\mathbb R}

\newcommand \Class{\mathcal C}

\newcommand{\Deck}{\mathrm{Deck}}
\newcommand{\Span}{\mathrm{Span}}
\newcommand{\Dim}{\mathrm{dim}}
\newcommand{\DK}{\mathcal{D}K}

\DeclareMathOperator \GL {GL}
\DeclareMathOperator \Or {O}
\DeclareMathOperator \RCD {\mathrm{RCD}}
\DeclareMathOperator \GH {d_{\mathrm{GH}}}

\DeclareMathOperator \D {\mathfrak{D}}
\DeclareMathOperator \CD {\mathrm{CD}}
\DeclareMathOperator \Iso {Iso}

\DeclareMathOperator \Diam {Diam}
\DeclareMathOperator \Mass {Mass}
\DeclareMathOperator \Spt {Spt}

\DeclareMathOperator \Dis {Dis}

\DeclareMathOperator \id {id}
\DeclareMathOperator \dist {\mathrm{d}}
\DeclareMathOperator \Fix {\mathrm{Fix}}

\DeclareMathOperator \m {\mathfrak{m}}

\DeclareMathOperator \diag {diag}

\DeclareMathOperator \dha {\dist^{\R^3}_{\mathrm{H}}}

\DeclareMathOperator \de {\dist_{\mathrm{E}}}

\DeclareMathOperator \tK {\tilde{\mathcal{K}}_{2\leq3}}
\DeclareMathOperator \Ks {\mathcal{K}^s_{2\leq3}}

\theoremstyle{definition}
\newtheorem{Definition}{Definition}[section]

\theoremstyle{definition}
\newtheorem{Theorem}[Definition]{Theorem}

\theoremstyle{definition}
\newtheorem{Proposition}[Definition]{Proposition}

\theoremstyle{definition}
\newtheorem{Notation}[Definition]{Notation}

\theoremstyle{definition}
\newtheorem{question}[Definition]{Question}

\theoremstyle{definition}

\theoremstyle{definition}
\newtheorem{Lemma}[Definition]{Lemma}

\theoremstyle{remark}
\newtheorem{Remark}[Definition]{Remark}

\begin{document}

\title{Contractibility of moduli spaces of {RCD}({0},{2})-structures}

\author{Dimitri Navarro}
\address{University of Oxford, Mathematical Institute}
\email{navarro@maths.ox.ac.uk}

\date{\today}

\begin{abstract}
This paper focuses on $\RCD(0,2)$-spaces, i.e. possibly non-smooth metric measure spaces with nonnegative Ricci curvature and dimension at most $2$. First, we establish a list of the compact topological spaces admitting an $\RCD(0,2)$-structure. We then describe the moduli space of $\RCD(0,2)$-structures for each space and show that it is contractible.\end{abstract}

\maketitle

\tableofcontents

\newpage

\section{Introduction}

A fundamental problem in Riemannian geometry is to study the existence of metrics satisfying a specific curvature constraint (e.g. nonnegative Ricci, scalar, or sectional curvature). When the existence problem finds a positive answer, it is interesting to describe the space of such metrics. The standard way to interpret this last problem is to study the topology of the moduli space associated with metrics satisfying the desired constraint. In the past two decades, moduli spaces of metrics with negative sectional curvature/positive scalar curvature have been studied (see \cite{Tuschmann-Wraith_15} for an introduction). More recently, results concerning moduli spaces of nonnegatively Ricci curved metrics have been established as well (see \cite{Tuschmann-Wiemeler_21}).\\
Instead of working with smooth metrics, it is also possible to study moduli spaces of singular metrics. Indeed, various synthetic definitions of curvature bounds have been introduced, such as Alexandrov spaces and $\RCD$-spaces, generalizing respectively lower bounds on the sectional and Ricci curvature. Lately, moduli spaces of metrics with nonnegative curvature in the Alexandrov sense have been studied in \cite{Belegradek_18}.\\
Here, we will focus on $\RCD(0,2)$-spaces, which can be seen as possibly singular metric measure spaces having dimension at most $2$ and nonnegative Ricci curvature, in a synthetic sense. This paper follows \cite{Mondino-Navarro_22}, which sets the foundations about moduli spaces of $\RCD(0,N)$-structures on compact topological spaces. We will focus on the following two questions.

\begin{question}\label{question: (i)}Up to homeomorphism, which compact topological spaces admit an $\RCD(0,2)$-structure?
\end{question}

\begin{question}\label{question: (ii)}Let $X$ be a compact topological space that admits an $\RCD(0,2)$-structure. What can be said about the topology of the moduli space of $\RCD(0,2)$-structures on $X$?
\end{question}

\subsection{Basic definitions}

Before presenting the main results, let us recall some basic definitions.

\begin{Definition}Let $(X,\dist,\m)$ be a triple where $(X,\dist)$ is a metric space and $\m$ is a measure on $X$. We say that $(X,\dist,\m)$ is a metric measure space (m.m.s. for short) when $(X,\dist)$ is a complete separable metric space and $\m$ is a boundedly finite Radon measure on $(X,\dist)$.
\end{Definition}

$\CD$-spaces (which generalize lower bounds on the Ricci curvature) were introduced independently by Lott and Villani in \cite{Lott-Villani_09}, and by Sturm in \cite{Sturm_I_06} and \cite{Sturm_II_06}. For simplicity, we only define $\CD(0,N)$-spaces (following Definition 1.3 in \cite{Sturm_II_06}). An extensive study of $\CD$-spaces is given in chapters 29 and 30 of \cite{Villani_09}.

\begin{Definition}\label{CD(0,N)}Given $N\in[1,\infty)$, a \emph{$\CD(0,N)$-space} is a m.m.s. $(X,\dist,\m)$ such that $\m(X)>0$, $(X,\dist)$ is locally compact and geodesic, and the Renyi entropy with parameter $N$ associated to $\m$ is weakly convex on the $L^2$-Wasserstein space $(\mathcal{P}_2(X,\dist,\m),\mathcal{W}_2)$ of probability measures that are absolutely continuous w.r.t. $\m$ and with finite variance.
\end{Definition}

Some $\CD$-spaces are not Ricci limit spaces, e.g. non-Riemannian Finsler spaces are $\CD$-spaces. In order to single out the ``Riemannian" $\CD$-structures, Ambrosio, Gigli, and Savar\'{e} strengthened the definition of
$\CD$-spaces by introducing $\RCD$-spaces in \cite{Ambrosio-Gigli-Savar_14} (see also \cite{Ambrosio-Gigli-Mondino-Rajala_15}, \cite{Gigli_15}, \cite{Ambrosio-Mondino-Savar_19}, and \cite{Cavalletti-Milman_21}). To this aim, a metric measure space is said to be infinitesimally Hilbertian if the Sobolev space $W^{1,2}$ is a Hilbert space (note that for a
Finsler manifold, $W^{1,2}$ is a Banach space).

\begin{Definition}Given $N\in[1,\infty)$, a m.m.s. $(X,\dist,\m)$ is an \emph{$\RCD(0,N)$-space} if it is an infinitesimally Hilbertian $\CD(0,N)$-space.
\end{Definition}

We now introduce the moduli space of $\RCD(0,N)$-structures on a topological space.

\begin{Definition}Let $X$ be a compact topological space and let $N\in[1,\infty)$. An \emph{$\RCD(0,N)$-structure on $X$} is an $\RCD(0,N)$-space $(X,\dist,\m)$ such that $\dist$ metrizes the topology of $X$ and $\Spt(\m)=X$. The \emph{moduli space of $\RCD(0,N)$-structures on $X$} is the set $\M_{0,N}(X)$ of $\RCD(0,N)$-structures on $X$ quotiented by measure-preserving isometries. $\M_{0,N}(X)$ is endowed with the mGH topology (see Remark \ref{remark: equivariant GH topology}).
\end{Definition}


\subsection{Main results}

The following notation introduces the topological spaces we are going to focus on.

\begin{Notation}\label{not: topological spaces}Let $\{*\}$, $I$, and $\mathbb{S}^1$ denote the singleton, the closed unit interval, and the unit circle, respectively. We denote $\mathbb{S}^2$, $\R\mathbb{P}^2$, $\mathbb{D}$, $\mathbb{M}^2$, $\mathbb{T}^2$ and $\mathbb{K}^2$ the $2$-sphere, the projective plane, the closed $2$-disc, the M\"obius band, the $2$-torus and the Klein bottle, respectively.
\end{Notation}

The following result answers Question \ref{question: (i)}.

\begin{Theorem}\label{prop: main}If $X$ is a compact topological space, then $X$ admits an $\RCD(0,2)$-structure if and only if it is homeomorphic to one of the following spaces: $\{*\}$, $I$, $\mathbb{S}^1$, $\mathbb{T}^2$, $\mathbb{K}^2$, $\mathbb{M}^2$, $\mathbb{S}^1\times I$, $\mathbb{S}^2$, $\mathbb{R}\mathbb{P}^2$, or $\mathbb{D}$.
\end{Theorem}

We will sketch the proof of Theorem \ref{prop: main} in the next section, and provide a more detailed statement (see table \ref{propintro: topological obstructions}). The subsequent result provides a partial answer to Question \ref{question: (ii)}.

\begin{Theorem}\label{th: contractibility}If $X$ is a compact topological space that admits an $\RCD(0,2)$-structure, then the moduli space $\M_{0,2}(X)$ of $\RCD(0,2)$-structures on $X$ is contractible.
\end{Theorem}

In the next section, we will explain how to prove Theorem \ref{th: contractibility}. In particular, we will describe the moduli space of $\RCD(0,2)$-structures on each of the topological spaces appearing in Notation \ref{not: topological spaces} (see Theorem \ref{th: contractibility 2}).

\begin{Remark}Theorem \ref{th: contractibility} should be compared with Theorem C in \cite{Mondino-Navarro_22}. Indeed, the result just mentioned shows in particular that the moduli spaces $\M_{0,N+4}(\mathbb{S}^N\times \mathbb{T}^4)$ have non-trivial higher rational cohomology groups; in particular, they are not contractible.
\end{Remark}

\subsection{Sketch of the proofs}

Let us briefly explain the idea behind the proof of Theorem \ref{prop: main}. Assume that $X$ is a compact topological space that admits an $\RCD(0,2)$-structure $(X,\dist,\m)$. We denote $p\colon \tilde{X}\to X$ the universal cover of $X$ (whose existence is granted by Theorem 1.1 in \cite{Mondino-Wei_19}) and we denote $(\tilde{X},\tilde{\dist},\tilde{\m})$ the lift of $(X,\dist,\m)$ to $\tilde{X}$ (we recall this notion in Section \ref{sec2}). Thanks to the structure theory of $\RCD$-spaces (which we recall briefly in section \ref{sec4}), it is possible to associate a dimension to $(X,\dist,\m)$, denoted $\dim(X,\dist,\m)$. Here, the dimension is bounded from above by $2$; therefore, we necessarily have $\dim(X,\dist,\m)\in\{0,1,2\}$.\\
We can take care of the case $\dim(X,\dist,\m)=0$ by relying on the structure theory of $\RCD$-spaces.\\
We can then treat the dimension $1$ case using the results of \cite{Kitabeppu-Lakzian_16} (which classify low-dimensional $\RCD$-spaces).\\
Finally, to treat the case $\dim(X,\dist,\m)=2$, we fix a splitting:
\begin{equation}\label{eq: intro split}
\phi\colon(\tilde{X},\tilde{\dist},\tilde{\m})\to(\overline{X},\overline{\dist},\overline{\m})\times\R^k,
\end{equation}
associated to the lift of $(X,\dist,\m)$ (we recall this notion in Section \ref{sec3}). The integer $k$ in \eqref{eq: intro split} is equal to the splitting degree $k(X)$ (see \eqref{eq: splitting degree}). In particular, we have $\dim(\overline{X},\overline{\dist},\overline{\m})=2-k(X)$. We can easily take care of the case $k(X)\in\{1,2\}$ using the first part of the proof. The final case $k(X)=0$ implies that $(\overline{X},\overline{\dist},\overline{\m})$ is a simply connected compact topological surface with boundary (using the results of \cite{Lytchak-Stadler_22}). It will then be easy to conclude the proof.\\
Thanks to what we previously mentioned, we will be able to prove that if $(X,\dist,\m)$ is an RCD(0, 2)-structure on a compact topological space $X$, then we have the following case disjunction:
\begin{table}[htb]
  \centering
  \begin{adjustbox}{max width=\textwidth}
    \begin{tabular}{|c||c|c|c|c|c|c|} \hline
$\dim(X,\dist,\m)$ & $0$ & $1$ & $1$ & $2$ & $2$ & $2$ \\ \hline
$k(X)$ & $0$ & $0$ & $1$ & $0$ & $1$ & $2$\\ \hline
$X\mbox{ is homeomorphic to}$ & $\{*\}$ & $I$ & $\mathbb{S}^1$ & $\mathbb{S}^2\mbox{, }\R\mathbb{P}^2\mbox{ or }\mathbb{D}$ & $I\times\mathbb{S}^1\mbox{ or }\mathbb{M}^2$ & $\mathbb{T}^2\mbox{ or }\mathbb{K}^2$\\ \hline
    \end{tabular}
  \end{adjustbox}
    \caption{}
  \label{propintro: topological obstructions}
\end{table}
\FloatBarrier

To prove Theorem \ref{th: contractibility}, we are going to study the moduli spaces $\M_{0,2}(X)$, where $X$ is any of the topological spaces appearing in the third row of table \ref{propintro: topological obstructions}.\\

\paragraph{\textbf{The singleton} $\{*\}$}

The moduli space $\M_{0,N}(\{*\})$ ($N\in[1,\infty]$) is obviously homeomorphic to $\R$, where the scale parameter $\R$ corresponds to the total measure $\m(\{*\})$.\\

\paragraph{\textbf{The unit interval} $I$}

Using results of \cite{Cavalletti-Milman_21}, we show in Proposition \ref{prop9.3} that $\M_{0,2}(I)$ is homeomorphic to $\R\times\{\mathcal{C}^*/\{\pm1\}\}$ (where $\mathcal{C}^*/\{\pm1\}$ is a quotient of the space of concave functions on $I$, it will be introduced in Notation \ref{not: concave}). Here, the $\R$ factor parametrizes the length of the interval, while the factor $\mathcal{C}^*/{\pm1}$ parametrizes the space of admissible measures.\\

\paragraph{\textbf{The circle} $\mathbb{S}^1$, \textbf{the} $2$-\textbf{torus} $\mathbb{T}^2$, \textbf{and the Klein bottle} $\mathbb{K}^2$}

Applying results from \cite{Bettiol-Derdzinski-Piccione_18}, \cite{Honda_20}, and \cite{DePhilippis-Gigli_18}, we provide a description of moduli spaces of $\RCD(0,N)$-structures on any closed flat manifold (see Proposition \ref{proposition: the case of flat manifolds}). Using the result just mentioned, we show that $\M_{0,2}(\mathbb{S}^1)$, $\M_{0,2}(\mathbb{T}^2)$, and $\M_{0,2}(\mathbb{K}^2)$ are homeomorphic to $\R^2$, $\R^4$, and $\R^3$, respectively (see propositions \ref{prop: circle}, \ref{prop: torus}, and \ref{prop: klein}).\\

\paragraph{\textbf{The M\"{o}bius band} $\mathbb{M}^2$ \textbf{and the cylinder} $\mathbb{S}^1\times I$}

It is possible to treat both spaces in a similar way. Focusing on the cylinder, we can show that any $\RCD(0,2)$-structure $(I\times\mathbb{S}^1,\dist,\m)$ on $I\times\mathbb{S}^1$ is isomorphic to $\mathcal{S}(I\times\mathbb{S}^1,\dist,\m)\times(\mathcal{A}(I\times\mathbb{S}^1,\dist,\m),\mathcal{H}^1)$, where $\mathcal{A}$ and $\mathcal{S}$ are the Albanese and soul maps (which reflect how
structures on the universal cover split, their definitions are recalled in Section \ref{sec3}). In this case, $\mathcal{S}(I\times\mathbb{S}^1,\dist,\m)$ is an $\RCD(0,1)$-structure on $I$ and $\mathcal{A}(I\times\mathbb{S}^1,\dist,\m)$ is a flat metric on $\mathbb{S}^1$. Therefore, using the continuity of the Albanese and soul maps (a result that we recall in Section \ref{sec3}), and using propositions \ref{prop: circle} and \ref{prop9.3} (which classify $\RCD(0,N)$-structures on $\mathbb{S}^1$ and $I$, respectively), we can conclude that $\M_{0,2}(I\times\mathbb{S}^1)$ is homeomorphic to $\R^3$ (see Proposition \ref{prop: cylinder}). Applying the same ideas, we can show that $\M_{0,2}(\mathbb{M}^2)$ is homeomorphic to $\R^3$ (see Proposition \ref{prop: mobius})\\

Now, observe that the remaining cases are $\mathbb{S}^2$, $\R\mathbb{P}^2$, and $\mathbb{D}$; in particular, they are all compact topological surfaces with boundary (possibly empty). Moreover, given a compact topological surface $X$, we can apply results from \cite{Honda_20}, \cite{DePhilippis-Gigli_18}, and \cite{Lytchak-Stadler_22} to show that $\M_{0,2}(X)$ is homeomorphic to $\R_{>0}\times\mathscr{M}_{\mathrm{curv}\geq0}(X)$, where $\mathscr{M}_{\mathrm{curv}\geq0}(X)$ is the moduli space of metrics on $X$ that are nonnegatively curved in the Alexandrov sense (see Lemma \ref{lem: case of surfaces}). In particular, to conclude, we only need to describe $\mathscr{M}_{\mathrm{curv}\geq0}(\mathbb{S}^2)$, $\mathscr{M}_{\mathrm{curv}\geq0}(\R\mathbb{P}^2)$, and $\mathscr{M}_{\mathrm{curv}\geq0}(\mathbb{D})$.
\\

\paragraph{\textbf{The} $2$-\textbf{sphere} $\mathbb{S}^2$}

It has been shown by Alexandrov that if $\dist$ is a nonnegatively curved metric on $\mathbb{S}^2$, then $(\mathbb{S}^2,\dist)$ is either isometric to the boundary $\partial B$ of a $3$-dimensional convex body $B$, or to the double $\mathcal{D}K$ of a $2$-dimensional convex body $K$ (see \cite{Kutateladze_05}). In particular, thanks to the result just mentioned, there is a surjective map $\Psi_{\mathbb{S}^2}\colon\mathcal{K}^s_{2\leq3}/\Or_3(\R)\to\mathscr{M}_{\mathrm{curv}\geq0}(\mathbb{S}^2)$, where $\mathcal{K}^s_{2\leq3}$ is the set of compact convex subsets of $\R^3$ with dimension $2$ or $3$ and Steiner point at the origin (see Notation \ref{not: Spaces of convex compacta}). In \cite{Belegradek_18}, Belegradek shows that $\Psi_{\mathbb{S}^2}$ is a homeomorphism. Therefore, the moduli space $\M_{0,2}(\mathbb{S}^2)$ is homeomorphic to $\R\times\{\mathcal{K}^s_{2\leq3}/\Or_3(\R)\}$ (see Proposition \ref{prop: contractibility S2}).\\

\paragraph{\textbf{The projective plane} $\R\mathbb{P}^2$}

Since $\mathbb{S}^2$ is the universal cover of $\R\mathbb{P}^2$, there is a homeomorphism between $\mathscr{M}_{\mathrm{curv}\geq0}(\R\mathbb{P}^2)$ and the moduli space $\mathscr{M}_{\mathrm{curv}\geq0}^{\mathrm{eq}}(\mathbb{S}^2)$ of nonnegatively curved metrics on $\mathbb{S}^2$ that are equivariant w.r.t. the antipodal map (we recall this result in Section \ref{sec2}). Now, let us denote $\tK$ the subspace of $\mathcal{K}^s_{2\leq3}$ whose elements are symmetric w.r.t. the origin (see Notation \ref{not: Spaces of convex compacta}). Observe that given $K,B\in\tK$ with respective dimensions $2$ and $3$, then $\mathcal{D}K$ and $\partial B$ can be seen as equivariant nonnegatively curved metrics on $\mathbb{S}^2$. Thanks to the fact just mentioned, we construct a well defined map $\Psi_{\mathbb{S}^2}^{\mathrm{eq}}\colon\tK/\Or_3(\R)\to\mathscr{M}_{\mathrm{curv}\geq0}^{\mathrm{eq}}(\mathbb{S}^2)$ (see \eqref{eq: homeo S2 eq}). We show in Proposition \ref{prop: realisation RP2} that $\Psi_{\mathbb{S}^2}^{\mathrm{eq}}$ is a $1$-$1$ correspondence; hence proving a realisation result for nonnegatively curved metrics on $\R\mathbb{P}^2$. Finally, in Proposition \ref{prop: RP2 continuity}, we show that $\Psi_{\mathbb{S}^2}^{\mathrm{eq}}$ is a homeomorphism. The proof of the proposition just mentioned is one of the most technical; let us sketch the proof of the direct part. To prove that $\Psi_{\mathbb{S}^2}^{\mathrm{eq}}$ is continuous, we fix a sequence $D_n\to D_{\infty}$ in $\tK$, and need to consider the following three cases:
\begin{itemize}
\item[($3$ to $3$)]$\dim(D_n)=3$ for every $n\in\N\cup\{\infty\}$,
\item[($3$ to $2$)]$\dim(D_n)=3$ for every $n\in\N$, and $\dim(D_{\infty})=2$,
\item[($2$ to $2$)]$\dim(D_n)=2$ for every $n\in\N\cup\{\infty\}$;
\end{itemize}
the goal being to show that $\Psi_{\mathbb{S}^2}^{\mathrm{eq}}([D_n])\to\Psi_{\mathbb{S}^2}^{\mathrm{eq}}([D_{\infty}])$ in the equivariant mGH topology (see Proposition \ref{prop: equivariant mGH}). We prove the convergence thanks to the approximation lemmas \ref{lem: 3 to 3}, \ref{lem: 3 to 2}, and \ref{lem: 2 to 2}; these results provide explicit approximations between spaces with various dimensions and give upper bounds on their distortions. As a result, we obtain that $\M_{0,2}(\R\mathbb{P}^2)$ is homeomorphic to $\R\times\{\tK/\Or_3(\R)\}$.\\

\paragraph{\textbf{The closed disc} $\mathbb{D}$}

This last case is similar to the previous one, but slightly more subtle. First of all, given $\alpha\in\mathbb{S}^2$, we denote $\mathcal{K}^{\alpha}_{2\leq3}$ the subset of $\mathcal{K}^s_{2\leq3}$ whose elements are symmetric w.r.t. $\{\alpha\}^{\perp}$. We then denote $\mathscr{K}_{2\leq3}\coloneqq\cup_{\alpha\in\mathbb{S}^2}\mathcal{K}^{\alpha}_{2\leq3}\times\{\alpha\}\subset \mathcal{K}^s_{2\leq3}\times\mathbb{S}^2$ (see Notation \ref{not: moduli space of the closed disc} for more details). Now, assume that $(D,\alpha)\in\mathscr{K}_{2\leq3}$, and note that only the following case can happen:
\begin{itemize}
\item[(i)]$\dim(D)=3$,
\item[(ii)]$\dim(D)=2$ and $\alpha\perp \Span(D)$,
\item[(iii)]$\dim(D)=2$ and $\alpha\in\Span(D)$.
\end{itemize}
In case (i), we denote $\Phi_{\mathbb{D}}(D,\alpha)\coloneqq\partial D\cap H_{\alpha}^+$ (where $H_{\alpha}^+$ is the upper half-space associated with $\alpha$). In case (ii), we write $\Phi_{\mathbb{D}}(D,\alpha)\coloneqq D$. In case (iii), we define $\Phi_{\mathbb{D}}(D,\alpha)\coloneqq\cup_{{\mathbf{i}}={\mathbf{1}}}^{\mathbf{2}}(D\cap H_{\alpha}^{+})^{\mathbf{i}}\subset\mathcal{D}D$ (see Notation \ref{not: DKalpha}). In every case, $\Phi_{\mathbb{D}}(D,\alpha)$ can be seen as a nonnegatively curved metric on $\mathbb{D}$. Therefore, we introduce in \eqref{eq: homeo disc} a well defined map $\Psi_{\mathbb{D}}\colon\mathscr{K}_{2\leq3}/\Or_3(\R)\to\mathscr{M}_{\mathrm{curv\geq0}}(\mathbb{D})$. In Proposition \ref{prop: realisation disc}, we show that $\Psi_{\mathbb{D}}$ is a $1$-$1$ correspondence; hence proving a realisation result for nonnegatively curved metrics on the disc. Finally, using the approximation lemmas of section \ref{sec: Approximation Lemmas}, we show in Proposition \ref{prop: disc continuity} that $\Psi_{\mathbb{D}}$ is a homeomorphism. Hence, we can conclude that $\M_{0,2}(\mathbb{D})$ is homeomorphic to $\R\times\{\mathscr{K}_{2\leq3}/\Or_3(\R)\}$.\\

To conclude, the discussion above leads to the following theorem.

\begin{Theorem}\label{th: contractibility 2}The following table describes moduli spaces of $\RCD(0,2)$-structures on compact topological spaces:
\begin{table}[htb]
  \centering
  \begin{adjustbox}{max width=\textwidth}
    \begin{tabular}{|c|c|} \hline
$X\mbox{ is homeomorphic to:}$ & $\M_{0,2}(X)\mbox{ is homeomorphic to:}$\\ \hline\hline
$\{*\}$ & $\R$ \\ \hline
$I$ & $\R\times\{\mathcal{C}^*/\{\pm1\}\}\mbox{ (see Notation \ref{not: concave})}$ \\ \hline
$\mathbb{S}^1$& $\R^2$\\ \hline
$\mathbb{T}^2$ & $\R^4$
\\ \hline
$\mathbb{S}^1\times I,\ \mathbb{M}^2,\mbox{ or }\mathbb{K}^2$ & $\R^3$\\ \hline
$\mathbb{S}^2$ & $\R\times\{\mathcal{K}^s_{2\leq3}/\Or_3(\R)\}\mbox{ (see Notation \ref{not: Spaces of convex compacta})}$ \\ \hline
$\R\mathbb{P}^2$ & $\R\times\{\tK/\Or_3(\R)\}\mbox{ (see Notation \ref{not: Spaces of convex compacta})}$\\ \hline
$\mathbb{D}$ & $\R\times \{\mathscr{K}_{2\leq3}/\Or_3(\R)\}\mbox{ (see Notation \ref{not: moduli space of the closed disc})}$\\ \hline
    \end{tabular}
  \end{adjustbox}
    \caption{}
  \label{propintrobis: topological obstructions}
\end{table}
\FloatBarrier
\end{Theorem}

Thanks to Theorem \ref{th: contractibility 2}, and proceeding via a case by case study (see propositions \ref{prop: circle}, \ref{prop9.3}, \ref{prop: torus}, \ref{prop: klein}, \ref{prop: mobius}, \ref{prop: cylinder}, \ref{prop: contractibility S2}, \ref{prop: projective plane}, and \ref{prop: closed disc}), we obtain Theorem \ref{th: contractibility}.

\subsection*{Organisation of the paper} In Section \ref{sec2}, we recall the equivariant mGH topology and relate $\RCD(0,N)$-structures on a compact topological space to equivariant structures on its universal cover. In Section \ref{sec3}, we recall the notion of splitting and show how to use it to construct the Albanese and soul maps, which will be fundamental to compute $\M_{0,2}(\mathbb{S}^1\times I)$ and $\M_{0,2}(\mathbb{M}^2)$. In Section \ref{sec4}, we prove Theorem \ref{prop: main}. The rest of the paper is devoted to the proof of Theorem \ref{th: contractibility}, which will be done throughout a case by case study following Theorem \ref{prop: main}. First of all, we introduce some lemmas in Section \ref{sec5} to simplify the computations in the case of flat manifolds and surfaces. In Section \ref{sec6}, we treat the $1$-dimensional case, i.e. we compute the moduli spaces of $\RCD(0,2)$-structures on $I$ and $\mathbb{S}^1$. Finally, in Section \ref{sec7}, we treat the rest of the cases. The most technical cases are those of $\mathbb{S}^2$, $\R\mathbb{P}^2$ and $\mathbb{D}$; these will be treated in Section \ref{sec: N=2 k(X)=0}.

\subsection*{Acknowledgements}

The author is supported by the European Research Council (ERC), under the European's Union Horizon 2020 research and innovation programme, via the ERC Starting Grant ``CURVATURE'', grant agreement No. 802689.\\
The author would like to thank Andrea Mondino for his constant support and for his comments on the paper. He is also grateful to G\'{e}rard Besson for exciting discussions on the subject. The author is indebted to the reviewer for their thorough reading and pertinent comments.

\section{Preliminaries}\label{Preliminaries}

Throughout this part $N\in[1,\infty)$ is a fixed real number and $X$ is a compact topological space that admits an $\RCD(0,N)$-structure. Let $p\colon\tilde{X}\to X$ be the universal cover of $X$ (whose existence is granted by Theorem 1.1 in \cite{Mondino-Wei_19}) and denote:
\begin{equation}\label{eq: revised fundamental group}
\overline{\pi}_1(X)\coloneqq\Deck(p)
\end{equation}
its group of deck transformations, also called the \emph{revised fundamental group of $X$}. Thanks to Corollary 2.2 in \cite{Mondino-Navarro_22}, $\overline{\pi}_1(X)$ is a finitely generated group with polynomial growth. We denote:
\begin{equation}\label{eq: splitting degree}
k(X)\coloneqq\text{polynomial growth order of }\overline{\pi}_1(X)\in\N\cap[0,N].
\end{equation}
We call $k(X)$ the \emph{splitting degree of $X$}.

\subsection{Equivariance}\label{sec2}

Later on, it will be convenient to see structures on $X$ as \emph{equivariant} structures on its universal cover. This section introduces moduli spaces of equivariant structures and the equivariant mGH topology.

\begin{Definition}[Equivariant $\RCD(0,N)$-structures]\label{def:equiv}An $\RCD(0,N)$-structure $(\tilde{X},\tilde{\dist},\tilde{\m})$ on $\tilde{X}$ is called equivariant if $\overline{\pi}_1(X)$ acts by isomorphisms on $(\tilde{X},\tilde{\dist},\tilde{\m})$.
\end{Definition}

\begin{Definition}[Equivariant isomorphism]Let $\tilde{\X}_i=(\tilde{X},\tilde{\dist}_i,\tilde{\m}_i)$ ($i\in\{1,2\}$) be an equivariant $\RCD(0,N)$-structure on $\tilde{X}$. We say that $\tilde{\X}_1$ and $\tilde{\X}_2$ are equivariantly isomorphic when there is an isomorphism $\phi$ of $\overline{\pi}_1(X)$ and an isomorphism $f\colon\tilde{\X}_1\to\tilde{\X}_2$ of m.m.s. such that $f(\gamma x)=\phi(\gamma)f(x)$, for every $\gamma\in\overline{\pi}_1(X)$, and every $x\in\tilde{X}$.
\end{Definition}

\begin{Definition}The \emph{moduli space of equivariant $\RCD(0,N)$-structures on $\tilde{X}$} is the set $\M_{0,N}^{\mathrm{eq}}(\tilde{X})$ of equivariant $\RCD(0,N)$-structures on $\tilde{X}$ quotiented by equivariant isomorphisms.
\end{Definition}

Before equipping $\M_{0,N}^{\mathrm{eq}}(\tilde{X})$ with a topology, let us first recall the distortion of a map.

\begin{Notation}\label{not: distortion}Assume that $f\colon(X_1,\dist_1)\to(X_2,\dist_2)$ is a map (not necessarily continuous) between metric spaces. The \emph{distortion of $f$} is defined as:
\begin{equation*}
\Dis(f)\coloneqq\sup\{\lvert \dist_2(f(x),f(y))-\dist_1(x,y)\rvert,x,y\in X_1\}.
\end{equation*}
It provides a measure of how close $f$ is to being an isometry.
\end{Notation}

We introduce the equivariant mGH pseudo-distance $\mathfrak{D}^{\mathrm{eq}}$ to compare equivariant $\RCD(0,N)$-structures on $X$.

\begin{Definition}\label{definition: equivariant mGH pseudo-distance}Let $\tilde{\X}_i=(\tilde{X},\tilde{\dist}_i,\tilde{\m}_i)$ ($i\in\{1,2\}$) be an equivariant $\RCD(0,N)$-structure on $\tilde{X}$ and let $\epsilon>0$. An \emph{equivariant mGH $\epsilon$-approximation} between $\tilde{\X}_1$ and $\tilde{\X}_2$ is a triple $(f,g,\phi)$ where $f\colon\tilde{X}\to\tilde{X}$ and $g\colon\tilde{X}\to\tilde{X}$ are Borel maps and $\phi$ is an isomorphism of $\overline{\pi}_1(X)$ such that:
\begin{itemize}
\item[(i)]$\max\{\Dis(f),\Dis(g)\}\leq\epsilon$ (see Notation \ref{not: distortion}),
\item[(ii)]for every $x\in\tilde{X}$, $\tilde{\dist}_1(g\circ f(x),x)\leq \epsilon$ and $\tilde{\dist}_2(f\circ g (x),x)\leq \epsilon$,
\item[(iii)]for every $\gamma\in\overline{\pi}_1(X)$ and $x\in\tilde{X}$, $f(\gamma x)=\phi(\gamma)f(x)$ and $g(\gamma x)=\phi^{-1}(\gamma)g(x)$,
\item[(iv)]$\max\{\dist_{\mathcal{P}}(f_*\tilde{\m}_1,\tilde{\m}_2),\dist_{\mathcal{P}}(g_*\tilde{\m}_2,\tilde{\m}_1)\}\leq\epsilon$,
\end{itemize}
where $\dist_{\mathcal{P}}$ denotes the Prokhorov distance. We define $\D^{\mathrm{eq}}(\tilde{\X}_1,\tilde{\X}_2)$ the \emph{equivariant mGH pseudo-distance between $\tilde{\X}_1$ and $\tilde{\X}_2$} as the minimum between $1/24$ and the infimum of all $\epsilon>0$ such that there exists an equivariant mGH $\epsilon$-approximation between $\tilde{\X}_1$ and $\tilde{\X}_2$.
\end{Definition}

The equivariant mGH pseudo-distance satisfies all the axioms of a distance apart from the triangle inequality. Indeed, given equivariant $\RCD(0,N)$-structures $\tilde{\X}_i=(\tilde{X},\tilde{\dist}_i,\tilde{\m}_i)$ ($i\in\{1,2,3\}$), we only have the following inequality:
\begin{equation}\label{eq: modified triangle inequality}
\D^{\mathrm{eq}}(\tilde{\X}_1,\tilde{\X}_3)\leq4(\D^{\mathrm{eq}}(\tilde{\X}_1,\tilde{\X}_2)+\D^{\mathrm{eq}}(\tilde{\X}_2,\tilde{\X}_3)),
\end{equation}
a proof of which is given in the Appendix of \cite{Mondino-Navarro_22}. Even though the equivariant mGH pseudo-distance $\D^{\mathrm{eq}}$ is a priori not a distance, it induces a metrizable topology as shown by the following proposition (see Proposition 2.8 of \cite{Mondino-Navarro_22} for a proof in the pointed case).

\begin{Proposition}\label{prop: equivariant mGH}The equivariant mGH pseudo-distance $\D^{\mathrm{eq}}$ induces a metrizable topology on $\M_{0,N}^{\mathrm{eq}}(\tilde{X})$, which we call the \emph{equivariant mGH topology}.
\end{Proposition}


\begin{Remark}\label{remark: equivariant GH topology}Later, we will sometimes forget about points (iii) and (iv) in Definition \ref{definition: equivariant mGH pseudo-distance}, leading to different notions of convergence:
\begin{itemize}
\item forgetting points (iii) and (iv) leads to the notion of \emph{GH $\epsilon$-approximation} and \emph{GH distance $\GH$},
\item forgetting points (iii) leads to the notion of \emph{mGH $\epsilon$-approximation} and \emph{mGH distance $\dist_{\mathrm{mGH}}$},
\item forgetting point (iv) leads to the notion of \emph{equivariant GH $\epsilon$-approximation} and \emph{equivariant GH distance $\mathcal{D}^{\mathrm{eq}}$}.
\end{itemize}
Moduli spaces of metrics, metric measure structures, and equivariant metrics will be respectively endowed with the topology induced by $\GH$ (GH topology), $\dist_{\mathrm{mGH}}$ (mGH topology) and $\mathcal{D}^{\mathrm{eq}}$ (equivariant GH topology).
\end{Remark}

To conclude this section, let us relate $\M_{0,N}(X)$ to $\M_{0,N}^{\mathrm{eq}}(\tilde{X})$. Assume that $(X,\dist,\m)$ is an $\RCD(0,N)$-structure on $X$. There exists a unique equivariant $\RCD(0,N)$-structure on $\tilde{X}$, called the \emph{lift of $(X,\dist,\m)$}, which we denote $p^*(X,\dist,\m)$, such that:
\begin{equation}\label{eq: lift}
p\colon p^*(X,\dist,\m)\to(X,\dist,\m)
\end{equation}
is a local isomorphism (see Corollary 2.1 of \cite{Mondino-Navarro_22}). Moreover, it is easily seen that isomorphic $\RCD(0,N)$-structures on $X$ have equivariantly isomorphic lifts. Therefore, there is a well defined map:
\begin{equation*}
p^*\colon\M_{0,N}(X)\to\M_{0,N}^{\mathrm{eq}}(\tilde{X})
\end{equation*}
called the \emph{lift map}, such that for every $[X,\dist,\m]\in\M_{0,N}(X)$, we have $p^*([X,\dist,\m])=[p^*(X,\dist,\m)]$. The next result is proved in the pointed case in \cite{Mondino-Navarro_22} (see Corollary A).

\begin{Theorem}\label{th: equiv homeo}If $\tilde{X}$ is compact, then the lift map $p^*\colon\M_{0,N}(X)\to\M_{0,N}^{\mathrm{eq}}(\tilde{X})$ is a homeomorphism.
\end{Theorem}

\subsection{Albanese variety and soul}\label{sec3}

A fundamental notion when studying $\RCD(0,N)$-spaces is the notion of splitting, which will be introduced in this section. We will also present the Albanese and soul maps, which will be important when computing moduli spaces in section \ref{section: N=2 k(X)=1}.

\begin{Notation}We denote $\de$ the Euclidean distance (it will always be clear which Euclidean space we discuss in the text).
\end{Notation}

Let $(X,\dist,\m)$ be an $\RCD(0,N)$-structure on $X$ with lift $(\tilde{X},\tilde{\dist},\tilde{\m})$ (see \eqref{eq: lift}). Thanks to Theorem 1.3 in \cite{Mondino-Wei_19} (after Theorem 1.4 of \cite{Gigli_13}), we can fix an isomorphism:
\begin{equation}\label{eq: splitting map}
\phi\colon(\tilde{X},\tilde{\dist},\tilde{\m})\to(\overline{X},\overline{\dist},\overline{\m})\times\R^k,
\end{equation}
where $k\in\N\cap[0,N]$, $\R^k$ is endowed with Euclidean distance $\de$ and Lebesgue measure $\mathcal{L}^k$, and $(\overline{X},\overline{\dist},\overline{\m})$ is a compact $\RCD(0,N-k)$-space with trivial revised fundamental group (see \eqref{eq: revised fundamental group}). Such a map is called a \emph{splitting of $(\tilde{X},\tilde{\dist},\tilde{\m})$}, $k$ is called the \emph{degree of $\phi$}, and $(\overline{X},\overline{\dist},\overline{\m})$ is called the \emph{soul of $\phi$}.

\begin{Remark}\label{rem: splitting degree}Thanks to Corollary 2.2 in \cite{Mondino-Navarro_22}, we have $k=k(X)$ (see \eqref{eq: splitting degree}). In particular, given any $\RCD(0,N)$-structure on $X$, any splitting of its lift would also have degree $k=k(X)$.
\end{Remark}

Since $\overline{X}$ is compact, an application of Lemma 1 in \cite{Shen-Wei_91} implies that the isomorphism group of $(\overline{X},\overline{\dist},\overline{\m})\times\R^k$ splits. Consequently, any isomorphism $T$ of $(\overline{X},\overline{\dist},\overline{\m})\times\R^k$ takes the form $T=(T_{S},T_{\R})$, where $T_{S}\in\Iso(\overline{X},\overline{\dist},\overline{\m})$ and $T_{\R}\in\Iso(\R^k)$. Hence, given $\gamma\in\overline{\pi}_1(X)$, we can introduce the following notations:
\begin{equation}\label{eq: euclidean and soul homomorphisms}
\phi_*(\gamma)\coloneqq \phi\gamma\phi^{-1}=(\{\phi\gamma\phi^{-1}\}_{S},\{\phi\gamma\phi^{-1}\}_{\R})\eqqcolon(\rho_{S}^{\phi}(\gamma),\rho_{\R}^{\phi}(\gamma)),
\end{equation}
where $\rho_{\R}^{\phi}\colon\overline{\pi}_1(X)\to\Iso(\R^k)$ and $\rho_{S}^{\phi}\colon\overline{\pi}_1(X)\to\Iso(\overline{X},\overline{\dist},\overline{\m})$ are called the \emph{Euclidean} and \emph{soul homomorphisms} associated to $\phi$, respectively. Throughout the paper, we will also use the following notation for the image of the Euclidean homomorphism:
\begin{equation}\label{eq: crystallographic group}
\Gamma(\phi)\coloneqq\Image(\rho_{\R}^{\phi})\subset\Iso(\R^k).
\end{equation}

\begin{Remark}\label{rem: crystallographic group}Thanks to Proposition 2.5 in \cite{Mondino-Navarro_22}, $\Gamma(\phi)$ is a crystallographic subgroup of $\Iso(\R^k)$ (i.e. $\Gamma(\phi)$ acts cocompactly and properly discontinuously on $\R^k$). This fact will be useful in order to answer Question \ref{question: (i)}.
\end{Remark}

Observe that one can associate a compact metric space $(\R^k/\Gamma(\phi),\dist_{\Gamma(\phi)})$ to $\phi$, where $\dist_{\Gamma(\phi)}$ is defined by:
\begin{equation}\label{eq: orbifold metric}
\forall x,y\in\R^{k}, \dist_{\Gamma(\phi)}([x],[y])\coloneqq\inf\{\de(x',y'),x'\in[x],y'\in[y]\}.
\end{equation}
Thanks to Lemma 2.1 in \cite{Mondino-Navarro_22}, the isometry class of $(\R^{k}/\Gamma(\phi),\dist_{\Gamma(\phi)})$ and the isomorphism class of $(\overline{X},\overline{\dist},\overline{\m})$ depend only on the isomorphism class of $(X,\dist,\m)$. Hence, we can define the \emph{Albanese variety of $[X,\dist,\m]$}:
\begin{equation}\label{eq: albanese variety}
\mathcal{A}([X,\dist,\m])\coloneqq[\R^k/\Gamma(\phi),\dist_{\Gamma(\phi)}]
\end{equation}
and the \emph{soul of $[X,\dist,\m]$}:
\begin{equation}\label{eq: soul}
\mathcal{S}([X,\dist,\m])\coloneqq[\overline{X},\overline{\dist},\overline{\m}].
\end{equation}

Let us conclude this section by recalling Theorem B in \cite{Mondino-Navarro_22}, which will be important in section \ref{section: N=2 k(X)=1}.

\begin{Theorem}\label{Th: Albanese and soul continuity}If $\mathcal{X}_n\to\mathcal{X}_{\infty}$ in $\M_{0,N}(X)$ in the mGH topology, then $\mathcal{A}(\mathcal{X}_n)\to\mathcal{A}(\mathcal{X}_{\infty})$ in the GH topology and $\mathcal{S}(\mathcal{X}_n)\to\mathcal{S}(\mathcal{X}_{\infty})$ in the mGH topology.
\end{Theorem}

\subsection{Essential dimension and topological obstructions}\label{sec4}

Given an $\RCD(0,N)$-structure $(X,\dist,\m)$ on $X$, there exists a unique $k\in\N\cap[0,N]$ such that the $k$-dimensional regular set $\mathcal{R}_k$ associated to $(X,\dist,\m)$ has positive $\m$-measure (see Theorem 0.1 in \cite{Bru-Semola_20}, after \cite{Mondino-Naber_19}). This integer $k$ is called the \emph{dimension of $(X,\dist,\m)$}, which we denote:
\begin{equation}\label{eq: dimension}
\dim(X,\dist,\m)\coloneqq k.
\end{equation}
Moreover, thanks to \cite{Kell-Mondino_18} (see also the independent proofs in \cite{DePhilippis-Marchese-Rindler_17} and \cite{Gigli-Pasqualetto_21}), $\m$ is absolutely continuous with respect to the $k$-dimensional Hausdorff measure $\mathcal{H}^k$ of $(X,\dist)$. Finally, if $k=N$, then there exists $a>0$ such that $\m=a\mathcal{H}^N$ (thanks to Corollary 1.3 in \cite{Honda_20}). We summarize this in the following proposition.

\begin{Proposition}\label{prop: semola-bru-honda}If $N\in[1,\infty)$ and $(X,\dist,\m)$ is an $\RCD(0,N)$-structure on a compact topological space $X$, then $\m$ is absolutely continuous with respect to $\mathcal{H}^k$ (where $k=\dim(X,\dist,\m)$). Moreover, if $k=N$, then there exists $a>0$ such that $\m=a\mathcal{H}^N$.
\end{Proposition}

We are now able to prove Theorem \ref{prop: main}.

\begin{proof}[Proof of Theorem \ref{prop: main}]The converse part is straightforward, so we will focus on the direct part. First of all, we assume that $\dim(X,\dist,\m)=0$. Thanks to Theorem 4.1 in \cite{Ambrosio-Honda-Tewodrose_18}, there is a measurable subset $\mathcal{R}_0^*\subset\mathcal{R}_0\subset X$ such that $\m$ is concentrated on $\mathcal{R}_0^*$ and such that $\m$ and $\mathcal{H}^0$ are absolutely continuous with respect to each other on $\mathcal{R}_0^*$. In particular, $\mathcal{R}_0^*\neq\varnothing$. Moreover, picking $x\in\mathcal{R}_0^*$, we have $\mathcal{H}^0(\{x\})=1$, hence $\m(\{x\})\neq0$. Thus, thanks to Corollary 30.9 in \cite{Villani_09}, $\m$ is a Dirac mass. In particular, since $\m$ has full support, $X$ is a singleton.\\
If $\dim(X,\dist,\m)=1$, then $\mathcal{R}_1\neq\varnothing$. Therefore, thanks to Theorem 1.1 in \cite{Kitabeppu-Lakzian_16}, $X$ is homeomorphic to either $\R$, $\R_{\geq0}$, $I$, or $\mathbb{S}^1$. However, since $X$ is compact, it is either homeomorphic to $I$ (if $k(X)=0$) or $\mathbb{S}^1$ (if $k(X)=1$).\\
From now on we assume that $\dim(X,\dist,\m)=2$. Let $(\tilde{X},\tilde{\dist},\tilde{\m})$ be the lift of $(X,\dist,\m)$ and $\phi\colon(\tilde{X},\tilde{\dist},\tilde{\m})\to(\overline{X},\overline{\dist},\overline{\m})\times\R^k$ be a splitting of $(\tilde{X},\tilde{\dist},\tilde{\m})$ (see \eqref{eq: splitting map}). Also, let us recall that $k=k(X)$ (see Remark \ref{rem: splitting degree}) and that $\dim(\tilde{X},\tilde{\dist},\tilde{\m})=2=k+\dim(\overline{X},\overline{\dist},\overline{\m})$ ($p$ being a local isomorphism).\\
If $k(X)=2$, then $\dim(\overline{X},\overline{\dist},\overline{\m})=0$; thus, $\overline{X}$ is a singleton $\{*\}$. In particular, $\rho_{\R}^{\phi}$ coincides with $\phi_*$ (see \eqref{eq: euclidean and soul homomorphisms}). Hence, $\phi$ induces a homeomorphism $X=\tilde{X}/\overline{\pi}_1(X)\simeq \R^2/\Gamma(\phi)$ (where $\Gamma(\phi)$ is defined in \eqref{eq: crystallographic group}). Moreover, $\overline{\pi}_1(X)$ acts freely on $X$; hence, $\Gamma(\phi)$ acts freely on $\R^2$. Therefore, being a crystallographic subgroup of $\Iso(\R^2)$ (see Remark \ref{rem: crystallographic group}), $\Gamma(\phi)$ is a Bieberbach subgroup of $\Iso(\R^2)$, i.e. it is a torsion free crystallographic group (see Proposition 1.1 in \cite{Charlap_86}). However, there are only two Bieberbach subgroups of $\Iso(\R^2)$ (up to isomorphism), leading respectively to $X$ homeomorphic to $\mathbb{T}^2$ or $\mathbb{K}^2$.\\
If $k(X)=1$, then $\dim(\overline{X},\overline{\dist},\overline{\m})=1$. Moreover, $(\overline{X},\overline{\dist},\overline{\m})$ is a compact $\RCD(0,1)$-space. Therefore, thanks to Proposition \ref{prop: semola-bru-honda}, there exists $a>0$ such that $\overline{\m}=a\mathcal{H}^1$. Also, $\overline{X}$ has trivial revised fundamental group. Hence, thanks to Theorem 1.1 in \cite{Kitabeppu-Lakzian_16}, there exists $r>0$ such that $(\overline{X},\overline{\dist})$ is isometric to $([0,r],\de)$. Therefore, we can assume that $(\overline{X},\overline{\dist},\overline{\m})=([0,1],r\de,a\mathcal{H}^1)$. Now, observe that $\Iso([0,1],r\de,a\mathcal{H}^1)\simeq\Z/{2\Z}$, where a generator is given by $s\colon t\to 1-t$. In particular, there are two cases, either $\Image(\rho_{S}^{\phi})=\{\id\}$ or $\Image(\rho_{S}^{\phi})\simeq\Z/2\Z$ (where $\rho_S^{\phi}$ is defined in \eqref{eq: euclidean and soul homomorphisms}).\\
First, let us suppose that $\Image(\rho_{S}^{\phi})=\{\id\}$. Observe that $\phi_*(\gamma)=(\id,\rho_{\R}^{\phi}(\gamma))$, for every $\gamma\in\overline{\pi}_1(X)$. Hence, $X$ is homeomorphic to $[0,1]\times \{\R/\Gamma(\phi)\}$. Moreover, $\Gamma(\phi)$ acts freely on $\R$, i.e. $\Gamma(\phi)$ is a Bieberbach subgroup of $\Iso(\R)$ (see Proposition 1.1 in \cite{Charlap_86}). However, $\Z$ is the only Bieberbach subgroup of $\Iso(\R)$ (up to isomorphism), which leads to $X$ homeomorphic to $[0,1]\times\mathbb{S}^1$.\\
Now, we assume that $\Image(\rho_{S}^{\phi})\simeq\Z/2\Z$. In particular, for every $\gamma\in\overline{\pi}_1(X)$, we have $\rho_S^{\phi}(\gamma)(1/2)=1/2$. Hence, $\Gamma(\phi)$ acts freely on $\R$. In particular, $\Gamma(\phi)$ is a Bieberbach subgroup of $\Iso(\R)$ (see Proposition 1.1 in \cite{Charlap_86}), i.e. it is conjugated to $\Z$ by an affine transformation. We then note that $\rho_{\R}^{\phi}$ is injective. Indeed, let us assume that $\rho_{\R}^{\phi}(\gamma)=0$, and, looking for a contradiction, assume that $\rho_{S}^{\phi}(\gamma)=s$. In that case, we have $\phi_{*}(\gamma)(1/2,t)=(1/2,t)$ for any fixed $t\in\R$, which is not possible as $\overline{\pi}_1(X)$ acts freely on $\tilde{X}$. Therefore, $\rho_{\R}^{\phi}$ is injective and $\overline{\pi}_1(X)$ is isomorphic to $\Z$. In conclusion, there is a unique generator $\gamma$ of $\overline{\pi}_1(X)$ such that $\phi_{*}(\gamma)(\overline{x},t)=(1-\overline{x},t+a)$ for some $a>0$ and every $(\overline{x},t)\in[0,1]\times\R$; therefore, $X$ is homeomorphic to $\mathbb{M}^2$.\\
If $k(X)=0$, then $\overline{\pi}_1(X)$ is finite. Moreover, thanks to Proposition \ref{prop: semola-bru-honda}, there exists $a>0$ such that $\m=a\mathcal{H}^2$. In particular, $(X,\dist,\mathcal{H}^2)$ is an $\RCD(0,2)$-space. As a result of Theorem 1.1 of \cite{Lytchak-Stadler_22}, $X$ is a topological surface with boundary (possibly empty), which implies $\pi_1(X)\simeq\overline{\pi}_1(X)$. Now, thanks to Theorem 5.1 and Theorem 10.1 in \cite{Massey_77} (which provide a classification of surfaces with boundary), $X$ is necessarily homeomorphic to either: a $2$-sphere with $k$ holes, an $m$-fold torus with $k$ holes, or an $m$-fold projective plane with $k$ holes (where $k$ is the number of path connected components of $\partial X$). In any case, $X$ can be represented as a polygon (as described in Section 10 of \cite{Massey_77}) and it is straightforward to compute its fundamental group using Van Kampen's theorem. As a consequence, we can see that the only topological surfaces with boundary having a finite fundamental group are the $2$-sphere, the $2$-sphere with $1$ hole (i.e. the disc), and the projective plane. This concludes the proof.
\end{proof}

%
%
%

\section{Moduli
spaces of closed flat manifolds and surfaces}\label{sec5}

In this section, we present two results (see Proposition \ref{proposition: the case of flat manifolds} and Lemma \ref{lem: case of surfaces}) which will be fundamental to describe the moduli spaces that we are interested in. The two results just mentioned will provide a partial description of the moduli spaces of $\RCD(0,N)$-structures on closed flat manifolds and topological surfaces, respectively.

\subsection{The case of closed flat manifolds}\label{sec: closed flat manifold}

Using \cite{Bettiol-Derdzinski-Piccione_18}, we are going to present a way to compute $\M_{0,N}(X)$ in the case where $X$ is homeomorphic to a closed flat manifold.

\begin{Definition}\label{Notation: crystallographic}Let $n\geq 1$ and let $\Gamma$ be a crystallographic subgroup of $\Iso(\R^n)$. We define:
\begin{itemize}
\item[(i)]$H_{\Gamma}\coloneqq\mathfrak{r}(\Gamma)\subset\Or_n(\R)$ (where $\mathfrak{r}(A,v)\coloneqq A$, for $A\in\GL_n(\R)$ and $v\in\R^n$),
\item[(ii)]$\mathcal{C}_{\Gamma}\coloneqq\{A\in\GL_n(\R),AH_{\Gamma}A^{-1}\subset\Or_n(\R)\}$,
\item[(iii)]$\mathcal{N}_{\Gamma}\coloneqq\mathfrak{r}(\mathrm{N}_{\Aff(\R^n)}(\Gamma))$ (where $\mathrm{N}_{\Aff(\R^n)}(\Gamma)$ is the normaliser of $\Gamma$ in $\Aff(\R^n)$).
\end{itemize}
The moduli space of flat metrics on $\R^n/\Gamma$ is the set $\mathscr{M}_{\mathrm{flat}}(\R^n/\Gamma)$ of flat Riemannian metrics on $\R^n/\Gamma$ quotiented by isometries. $\mathscr{M}_{\mathrm{flat}}(\R^n/\Gamma)$ is equipped with the GH topology (see Remark \ref{remark: equivariant GH topology}).
\end{Definition}

\begin{Proposition}\label{proposition: the case of flat manifolds}Let $n\geq 1$, let $\Gamma$ be a Bieberbach subgroup of $\Iso(\R^n)$, and let $N\in[1,\infty)$. If $N<n$, then there are no $\RCD(0,N)$-structures on $\R^n/\Gamma$. If $N\geq n$, then any $\RCD(0,N)$-structure on $\R^n/\Gamma$ is also an $\RCD(0,n)$-structure. Moreover, there exist homeomorphisms:
$$
\M_{0,n}(\R^n/\Gamma)\simeq\R_{>0}\times\mathscr{M}_{\mathrm{flat}}(\R^n/\Gamma)\simeq\R\times[\Or_n(\R)\backslash\mathcal{C}_{\Gamma}]/\mathcal{N}_{\Gamma},
$$
where the left action of $\Or_n(\R)$ on $\mathcal{C}_{\Gamma}$ is given by multiplication on the left and the right action of $\mathcal{N}_{\Gamma}$ on $[\Or_n(\R)\backslash\mathcal{C}_{\Gamma}]$ is defined by $[A]\cdot B\coloneqq[AB]$, given $[A]\in\Or_n(\R)\backslash\mathcal{C}_{\Gamma}$ and $B\in\mathcal{N}_{\Gamma}$.
\end{Proposition}

\begin{proof}We denote $X\coloneqq\R^n/\Gamma$. Let us show that, for $N<n$, there are no $\RCD(0,N)$-structures on $X$. Indeed, $\Gamma$ is a Bieberbach subgroup of $\Iso(\R^n)$, thus $X$ is a topological manifold and $\overline{\pi}_1(X)\simeq\pi_1(\R^n/\Gamma)\simeq\Gamma$. Hence, thanks to Bieberbach's first Theorem (see Theorem 3.1 in \cite{Charlap_86}), we have $k(X)=n$ (see \eqref{eq: splitting degree}). If $(X,\dist,\m)$ is an $\RCD(0,N)$-structure on $X$, then Remark \ref{rem: splitting degree} implies that the degree of any splitting is equal to $n$ and belong to $[0,N]$; hence, $n\leq N$.\\
Now, assume that $n\leq N$. Let $(X,\dist,\m)$ be an $\RCD(0,N)$-structure on $X$ and let us prove that it is an $\RCD(0,n)$-structure. We denote $(\R^n,\tilde{\dist},\tilde{\m})\coloneqq p^*(X,\dist,\m)$ the associated lift (where $p\colon \R^n\to\R^n/\Gamma=X$ is the quotient map) and we fix a splitting $\phi$ of $(\R^n,\tilde{\dist},\tilde{\m})$ with soul $(\overline{X},\overline{\dist},\overline{\m})$. Note that $\phi$ has degree $n$. Let us show that $\overline{X}$ is a singleton. Seeking for a contradiction, we assume that there exists $\overline{x},\overline{y}\in \overline{X}$ such that $\overline{x}\neq\overline{y}$. Let $\overline{\gamma}\colon[0,L]\to\overline{X}$ be a minimizing geodesic from $\overline{x}$ to $\overline{y}$, which is parametrized by arclength. Observe that $\phi$ induces an isometric embedding $\phi^{-1}\colon\overline{\gamma}(]0,L[)\times(\R^n,\de)\to(\R^n,\tilde{\dist})$. However, $\overline{\gamma}(]0,L[)\times\R^n$ is homeomorphic to $\R^{n+1}$. Hence, $\phi$ gives rise to a continuous injective map $f\colon \R^{n+1}\to\R^n$; but no such map exists (see Corollary 2B.7 in \cite{Hatcher_02}). In conclusion, $\overline{X}$ is a singleton $\{*\}$. Now, since $\overline{\m}$ has full support, there exists $a>0$ such that $\overline{\m}=a\delta_*$. Hence, $(\overline{X},\overline{\dist},\overline{\m})\times\R^n$ is isomorphic to $(\R^n,\de,a\mathcal{H}^n)$. Moreover, since $\overline{X}$ is a singleton, then $\rho^{\phi}_{\R}$ is injective and coincides with $\phi_*$ (see \eqref{eq: euclidean and soul homomorphisms}). Thus, $\Gamma(\phi)\simeq\Gamma$ and $\phi$ induces an isomorphism $(\R^n/\Gamma,\dist,\m)\simeq(\R^n/\Gamma(\phi),\dist_{\Gamma(\phi)},a\mathcal{H}^n)$. Now, observe that $\dist_{\Gamma(\phi)}$ and $\mathcal{H}^n$ are respectively the Riemannian distance and measure associated to $\R^n/\Gamma(\phi)$, which is flat of dimension $n$. Hence, $(\R^n/\Gamma(\phi),\dist_{\Gamma(\phi)},a\mathcal{H}^n_{\dist_{\Gamma(\phi)}})$ is an $\RCD(0,n)$-space and $(X,\dist,\m)$ as well (a fortiori).\\
Let us now prove that $\M_{0,n}(\R^n/\Gamma)\simeq\R\times[\Or_n(\R)\backslash\mathcal{C}_{\Gamma}]/\mathcal{N}_{\Gamma}$. We have shown above that if $(X,\dist,\m)$ is an $\RCD(0,n)$-structure on $X$, then $[X,\dist]\in\mathscr{M}_{\mathrm{flat}}(X)$ and there exists $a>0$ such that $\m=a\mathcal{H}^n$. In particular, the map $\Phi\colon\M_{0,n}(X)\to\mathscr{M}_{\text{flat}}(X)\times\R_{>0}$ defined by $\Phi([X,\dist,\m])\coloneqq([X,\dist],\m(X)/\mathcal{H}^n(X))$ is well defined. The map $\Psi\colon\mathscr{M}_{\text{flat}}(X)\times\R_{>0}\to\M_{0,n}(X)$ defined by $\Psi([X,\dist],a)\coloneqq[X,\dist,a\mathcal{H}^n]$ is also well defined and it is clear that $\Psi$ and $\Phi$ are respectively inverse to each other.\\
Let us show that $\Phi$ is continuous. Assume that $[X,\dist_k,\m_k]\to[X,\dist_{\infty},\m_{\infty}]$ in $\M_{0,n}(X)$ and, for $k\in\N\cup\{\infty\}$, let us denote $a_k\coloneqq\m_k(X)/\mathcal{H}^n(X)$. Observe that we necessarily have $[X,\dist_k]\to[X,\dist_{\infty}]$ in the GH topology. We then notice that $(X,\dist_{\infty})$ has Hausdorff dimension $n$; hence, thanks to Theorem 1.2 of \cite{DePhilippis-Gigli_18}, $[X,\dist_k,\mathcal{H}^n]\to[X,\dist_{\infty},\mathcal{H}^n]$ in the mGH topology. Therefore, $a_k\to a_{\infty}$; thus, $\Phi$ is continuous.\\
Conversely, assume that $[X,\dist_k]\to[X,\dist_{\infty}]$ in $\mathscr{M}_{\text{flat}}(X)$ and let $a_k\to a_{\infty}$ in $\R_{>0}$. Observe that, thanks to Theorem 1.2 of \cite{DePhilippis-Gigli_18}, $[X,\dist_k,\mathcal{H}^n]\to(X,\dist_{\infty},\mathcal{H}^n)$ in the mGH topology. Hence, $[X,\dist_k,a_k\mathcal{H}^n]\to[X,\dist_{\infty},a_{\infty}\mathcal{H}^n]$ in the mGH topology; thus, $\Psi$ is continuous.\\
Now, we have shown that $\M_{0,n}(\R^n/\Gamma)$ is homeomorphic to $\R_{>0}\times\mathscr{M}_{\text{flat}}(\R^n/\Gamma)$. In order to conclude, notice that, thanks to Proposition 4.3 of \cite{Bettiol-Derdzinski-Piccione_18}, $\mathscr{M}_{\text{flat}}(\R^n/\Gamma)$ is homeomorphic to $[\Or_n(\R)\backslash\mathcal{C}_{\Gamma}]/\mathcal{N}_{\Gamma}$.
\end{proof}

\begin{Remark}Given $n\geq1$ and $\Gamma$ a Bieberbach subgroup of $\Iso(\R^n)$, $\Or_n(\R)\backslash\mathcal{C}_{\Gamma}$ is homeomorphic to $\R^d$ for some $d\in\N$ (see Theorem B in \cite{Bettiol-Derdzinski-Piccione_18}). In particular, $\M_{0,N}(\R^n/\Gamma)$ is connected for every $N\geq n$.
\end{Remark}

\subsection{The case of surfaces}

Let $X$ be a compact topological surface (with or without boundary) that admits an $\RCD(0,2)$-structure and let $(X,\dist,\m)$ be an $\RCD(0,2)$-structure on $X$. Thanks to the proof of Theorem \ref{prop: main} (see table \ref{propintro: topological obstructions}), we necessarily have $\dim(X,\dist,\m)=2$. Therefore, applying Proposition \ref{prop: semola-bru-honda}, there exists $a>0$ such that $\m=a\mathcal{H}^2$. In particular, thanks to Theorem 1.1 of \cite{Lytchak-Stadler_22}, $(X,\dist)$ is an Alexandrov space with nonnegative curvature. Therefore, proceeding exactly as in the last part of the proof of Proposition \ref{proposition: the case of flat manifolds}, we obtain the following result.

\begin{Lemma}\label{lem: case of surfaces}If $X$ is a compact topological surface with boundary (possibly empty) that admits an $\RCD(0,2)$-structure and $p\colon\tilde X\to X$ denotes its universal cover, then the following map:
\begin{equation*}
[X,\dist,\m]\in\M_{0,2}(X)\to (\m(X)/\mathcal{H}^2(X),[X,\dist])\in\R_{>0}\times\mathscr{M}_{\mathrm{curv}\geq0}(X)
\end{equation*}
is a homeomorphism, where $\mathscr{M}_{\mathrm{curv}\geq0}(X)$ is the moduli space of nonnegatively curved metrics on $X$ in the Alexandrov sense (endowed with the GH topology). In addition, the same map induces a homeomorphism:
$$
\M_{0,2}^{\mathrm{eq}}(\tilde{X})\simeq\R_{>0}\times\mathscr{M}_{\mathrm{curv}\geq0}^{\mathrm{eq}}(\tilde{X}),
$$
where $\mathscr{M}_{\mathrm{curv}\geq0}^{\mathrm{eq}}(\tilde{X})$ is the moduli space of equivariant metrics on $\tilde{X}$ that are nonnegatively curved in the Alexandrov sense (endowed with the equivariant GH topology, see Remark \ref{remark: equivariant GH topology}). 
\end{Lemma}


\section{Moduli spaces in the 1-dimensional
case}\label{sec6}

\subsection{The circle}

\begin{Proposition}\label{prop: circle}The moduli space $\M_{0,2}(\mathbb{S}^1)$ of $\RCD(0,2)$-structures on $\mathbb{S}^1$ is homeomorphic to $\R^2$; in particular, it is contractible.\end{Proposition}

\begin{proof}
Given $n\geq1$ and $N\geq n$, Proposition \ref{proposition: the case of flat manifolds} implies that we have a homeomorphism:
\begin{equation}
\M_{0,N}(\R^n/\Z^n)\simeq\R\times[\Or_n(\R)\backslash\GL_n(\R)]/\GL_n(\Z).
\end{equation}
When $n=1$, we have $\Or_1(\R)=\GL_1(\Z)=\{\pm1\}$. In addition, $\GL_1(\R)=\R^*$ is commutative; hence, $\GL_1(\Z)$ acts trivially on $\Or_1(\R)\backslash\GL_1(\R)$. Therefore, $[\Or_1(\R)\backslash\GL_1(\R)]/\GL_1(\Z)$ is homeomorphic to $\{\pm1\}\backslash\R^*$, i.e. is homeomorphic to $\R_{>0}$, which is itself homeomorphic to $\R$. In conclusion, for every $N\geq1$, $\M_{0,N}(\mathbb{S}^1)$ is homeomorphic to $\R^2$, which concludes the proof.
\end{proof}

\begin{Remark}It is easily checked that $\Psi\colon[\mathbb{S}^1,\dist,\m]\in\M_{0,N}(\mathbb{S}^1)\to(\Diam(\mathbb{S}^1,\dist),\m(\mathbb{S}^1))\in\R_{>0}^2 $ is an explicit homeomorphism.
\end{Remark}

\subsection{The unit interval}\label{sec: the unit interval}

\begin{Notation}[Space of concave functions]\label{not: concave}We denote $\mathcal{C}^*$ the space of concave functions $f\colon I\to\R$ such that $f$ is strictly positive on $\mathrm{int}(I)$. We endow $\mathcal{C}^*$ with the topology of uniform convergence on compact subsets of $\mathrm{int}(I)$. The aforementioned topology is metrizable with the following distance:
$$
\dist_{\Class^*}(f,g)\coloneqq\sum_{k=0}^{\infty}2^{-k}\min\{1,\dist_k(f,g)\},
$$
where $f,g\in\Class^*$ and $\dist_k(f,g)\coloneqq\sup_{t\in[2^{-k},1-2^{-k}]}\{\lvert f(t)-g(t)\rvert\}$.\\
For every $f\in\Class^*$, we define $-1\cdot f(t)\coloneqq f(1-t)$, which gives rise to an action of $\{\pm1\}$ on $\Class^*$. We denote $\Class^*/\{\pm1\}$ the quotient of $\mathcal{C}^*$ by the action of $\{\pm1\}$, endowed with the quotient topology.
\end{Notation}

\begin{Remark}Observe that $\{\pm1\}$ acts by isometries on $(\Class^*,\dist_{\Class^*})$. Therefore, the distance $\dist_{\Class^*/\{\pm1\}}([f],[g])\coloneqq\min\{\dist_{\Class^*}(f,g),\dist_{\Class^*}(f,-1\cdot g)\}$ metrizes the topology of $\mathcal{C}^*/\{\pm1\}$.
\end{Remark}

\begin{Proposition}\label{prop9.3}The moduli space $\M_{0,1}(I)$ is homeomorphic to $\R^2$. Moreover, for every $N\in(1,\infty)$, the moduli space $\M_{0,N}(I)$ is homeomorphic to $\R\times\{\mathcal{C}^*/\{\pm1\}\}$ (which is contractible).
\end{Proposition}

\begin{proof}\underline{We start with the case $N=1$.} Let us first assume that $(I,\dist,\m)$ is an $\RCD(0,1)$-structure on $I$. Thanks to the proof of Theorem \ref{prop: main}, we necessarily have $\dim(X,\dist,\m)=1$. Therefore, thanks to Proposition \ref{prop: semola-bru-honda}, there exists $a>0$ such that $\m=a\mathcal{H}^1$. Moreover, thanks to Theorem 1.1 of \cite{Kitabeppu-Lakzian_16}, $(I,\dist)$ is isometric to $(I,L\de)$, where $L\coloneqq\Diam(I,\dist)$. Hence, the map $\Psi\colon\R_{>0}\times\R_{>0}\to\M_{0,1}(I)$ defined by $\Psi(L,a)\coloneqq[I,L\de,a\mathcal{H}^1]$ is surjective. It is then readily checked that $\Phi\colon\M_{0,1}(I)\to\R_{>0}\times\R_{>0}$ defined by $\Phi([I,\dist,\m])\coloneqq(\Diam(I,\dist),\m(I)/\mathcal{H}^1(I))$ is an inverse. Moreover, thanks to Theorem 1.2 of \cite{DePhilippis-Gigli_18}, we can prove that $\Phi$ is continuous (just proceeding the same way as in the last part of the proof of Proposition \ref{proposition: the case of flat manifolds}). Finally, the proof of $\Psi$'s continuity is trivial; thus, $\M_{0,1}(I)$ is homeomorphic to $\R^2$.\\
\underline{From now on we assume that $1<N$.} If $(I,\dist,\m)$ is an $\RCD(0,N)$-structure on $I$, then $(I,\dist,\m)$ is isomorphic to $(I,L\de,\m')$, where $L\coloneqq\Diam(I,\dist)$ and $\m'$ is a finite Radon measure on $I$. Thanks to Theorem A.2 of \cite{Cavalletti-Milman_21}, there exists $g\in\Class^*$ (see Notation \ref{not: concave}) such that $\m'=g^{N-1}\mathcal{L}^1$. Conversely, thanks to Theorem A.2 of \cite{Cavalletti-Milman_21}, for every $g\in\Class^*$ and $L>0$, $(I,L\de,g^{N-1}\mathcal{L}^1)$ is an $\RCD(0,N)$-structure on $I$. Now, if $g_1,g_2\in\Class^*$ satisfy $[I,\de,g_1^{N-1}\mathcal{L}^1]=[I,\de,g_2^{N-1}\mathcal{L}^1]$, then, there exists $\phi\in\Iso(I,\de)$ such that $g_2^{N-1}\mathcal{L}^1=\phi_*(g_1^{N-1}\mathcal{L}^1)=(g_1\circ\phi^{-1})^{N-1}\mathcal{L}^1$; hence, $g_1=g_2\circ\phi$. However, $\Iso(I,\de)$ consists of $2$ elements, the identity $\id_I$ and the symmetry $t\to 1-t$. Thus, $[g_1]=[g_2]\in\Class^*/\{\pm1\}$. Conversely, if $g_1,g_2\in\Class^*$ satisfy $[g_1]=[g_2]\in\Class^*/\{\pm1\}$, then $(I,\de,g_1^{N-1})$ is isomorphic to $(I,\de,g_2^{N-1})$. Therefore, the following two maps are well defined and respectively inverse to each other:
\begin{itemize}
\item $\Phi\colon\M_{0,N}(I)\to\R_{>0}\times\Class^*/\{\pm1\}$ defined by $\Phi([I,\dist,\m])\coloneqq(L,[g])$, where $L\coloneqq\Diam(I,\dist)$ and $g\in\Class^*$ satisfies $[I,\dist,\m]=[I,L\de,g^{N-1}\mathcal{L}^1]$,
\item $\Psi\colon\R_{>0}\times\Class^*/\{\pm1\}\to\M_{0,N}(I)$ defined by $\Psi(L,[g])\coloneqq[I,L\de,g^{N-1}\mathcal{L}^1]$.
\end{itemize}
\underline{First, we show that $\Phi$ is continuous.} Assume that $[I,\dist_n,\m_n]\to[I,\dist_{\infty},\m_{\infty}]$ in $\M_{0,N}(I)$ and, for every $n\in\N\cup\{\infty\}$, denote $(L_n,[g_n])\coloneqq\Phi([I,\dist_n,\m_n])$. Observe that $(I,\dist_n)$ converge to $(I,\dist_{\infty})$ in the GH topology; in particular, $L_n\to L_{\infty}$. Now, let us show that $[g_n]\to[g_{\infty}]$ in $\Class^*/\{\pm1\}$. Note that it is sufficient to prove that every subsequence of $\{g_n\}$ admits a subsequence converging to $\nu\cdot g_{\infty}$ for some $\nu\in\{\pm1\}$. We'll just show that $\{g_n\}$ admits a subsequence converging to $\nu\cdot g_{\infty}$ for some $\nu\in\{\pm1\}$ (the proof for a subsequence of $\{g_n\}$ being exactly the same).\\
Observe first that $\{g_n\}$ is uniformly bounded in $L^{\infty}(I)$. Indeed, $f_n\coloneqq g_n^{N-1}$ is a $\CD(0,N)$-density on $\mathrm{int}(I)$ (see Definition A.1 of \cite{Cavalletti-Milman_21}). Hence, for every $n\in\N$, we have $\sup_If_n\leq N\lvert f_n\rvert_{L^1}$ (see Lemma A.8 of \cite{Cavalletti-Milman_21}). Moreover, observe that $\lvert f_n\rvert_{L^1}=\m_n(I)\leq M$, where $M\coloneqq\sup_{n\in\N}\{\m_n(I)\}$. Hence, for every $n\in\N$, we obtain $
\rvert g_n\lvert_{L^{\infty}}\leq (NM)^{1/N-1}\eqqcolon\overline{M}$.\\
Now, observe that for every $\epsilon>0$, $\{g_n\}$ is equicontinuous on $I_{\epsilon}\coloneqq[\epsilon,1-\epsilon]$. Indeed, given $n\in\N$, $x\in I_{\epsilon}$, and $t\in[0,1]$, we have the following inequality:
\begin{equation*}
-\overline{M}/1-x\leq-g_n(x)/1-x\leq g_{n,r}'(x)\leq g_{n,l}'(x)\leq g_n(x)/x\leq\overline{M}/x,
\end{equation*}
since $g_n$ is concave and positive on $\mathrm{int}(I)$ (we denoted $g_{n,r}'$ and $g_{n,l}'$ the right and left derivatives of $g_n$, respectively). Hence, for every $x\in I_{\epsilon}$, we have $\mathrm{Lip}_x(g_n)\leq\overline{M}/\epsilon$. Therefore, $\{g_n\}$ is equicontinuous on $I_{\epsilon}$.\\
Now, passing to a subsequence if necessary, we can assume that $g_n$ is converging uniformly to some continuous function $g\colon(0,1)\to\R$ on every compact subset $K\subset(0,1)$ (applying Arzel\`{a}--Ascoli Theorem and a diagonal argument). Observe that $g$ is nonnegative and concave on $(0,1)$; hence, we can assume that $g$ is continuous on $I$. A fortiori, note that $\{g_n^{N-1}\}$ converges uniformly to $g^{N-1}$ on compact subsets of $(0,1)$. Let us prove that there exists $\nu\in\{\pm1\}$ such that $g=\nu\cdot g_{\infty}$.\\
First, observe that $g_n^{N-1}\mathcal{L}^1\to g^{N-1}\mathcal{L}^1$ in the weak-$*$ topology. Indeed, let us fix $f\in\Class^0(I)$ and let $\epsilon>0$. Denoting $\dist_{\epsilon}(g_n^{N-1},g^{N-1})\coloneqq\sup_{I_{\epsilon}}(\lvert g_n^{N-1}-g^{N-1} \rvert)$ and splitting the integral into three part, we easily obtain $\lvert \int_If(g_n^{N-1}-g^{N-1})\rvert\leq 4\epsilon\overline{M}^{N-1}\lvert f\rvert_{L^{\infty}} +\lvert f\rvert_{L^{\infty}}\dist_{\epsilon}(g_n^{N-1},g^{N-1})$. In particular, for every $\epsilon>0$, we have $\limsup_{n\to\infty}\lvert \int_If(g_n^{N-1}-g^{N-1})\rvert\leq4\epsilon\overline{M}^{N-1}\lvert f\rvert_{L^{\infty}}$. Hence, for every continuous function $f\in\Class^0(I)$, we obtain $\lim_{n\to\infty}\int_Ifg_n^{N-1}=\int_Ifg^{N-1}$. This implies that $\{(I,L_n\de,g_n^{N-1}\mathcal{L}^1)\}$ converges in the mGH topology to $(I,L_{\infty}\de,g^{N-1}\mathcal{L}^1)$. However, $\{(I,L_n\de,g_n^{N-1}\mathcal{L}^1)\}$ also converges to $(I,L_{\infty}\de,g_{\infty}^{N-1}\mathcal{L}^1)$ in the mGH topology. Thus, there exists an isometry $\phi\colon(I,L_{\infty}\de)\to(I,L_{\infty}\de)$ such that $\phi_*(g_{\infty}^{N-1}\mathcal{L}^1)=g^{N-1}\mathcal{L}^1$, i.e. $g=g_{\infty}\circ\phi^{-1}$. However, $\Iso(I,L_{\infty}\de)$ consists of two elements, the identity $\id_I$ and the symmetry with center $1/2$. Thus, $g=\nu\cdot g_{\infty}$ for some $\nu\in\{\pm1\}$, which concludes the proof of $\Phi$'s continuity.\\
\underline{Now, we are going to show that $\Psi=\Phi^{-1}$ is continuous.} Assume that $\{L_n,[g_n]\}$ converges to $(L_{\infty},[g_{\infty}])$ in $\R_{>0}\times\{\Class^*/\{\pm1\}\}$. This implies that there exists a sequence $\{\nu_n\}$ in $\{\pm1\}$ such that $\{\nu_n\cdot g_n\}$ converges to $g_{\infty}$ uniformly on compact subsets of $(0,1)$. Let us denote $\tilde{g}_n\coloneqq\nu_n\cdot g_n$ and observe that, for every $n\in\N$, $(I,L_n,g_n^{N-1}\mathcal{L}^1)$ and $(I,L_n,\tilde{g}_n^{N-1}\mathcal{L}^1)$ are isomorphic. We need to show that $\{(I,L_n,\tilde{g}_n^{N-1}\mathcal{L}^1)\}$ converges to $(I,L_{\infty},g_{\infty}^{N-1}\mathcal{L}^1)$ in the mGH topology. Since $\lvert L_n-L_{\infty}\rvert\to0$, it is sufficient to show that $\tilde{g}_n^{N-1}\mathcal{L}^1$ converges to $g_{\infty}^{N-1}\mathcal{L}^1$ in the weak-$*$ topology. Moreover, proceeding exactly as in the last paragraph, it is enough to prove that $\{\tilde{g}_n\}$ is uniformly bounded in $L^{\infty}(I)$. Let $n\in\N$ and observe that, thanks to the nonnegativity and concavity of $\tilde{g}_n$, we have the following three cases:
\begin{itemize}
\item if $x<1/4$, we have $\tilde{g}_n(x)\leq\tilde{g}_n(1/2)+2(1-2x)(\tilde{g}_n(1/4)-\tilde{g}_n(1/2))$, which implies that $\tilde{g}_n(x)\leq3\max_{[1/4,3/4]}\{\tilde{g}_n\}$,
\item if $x>3/4$, we have $\tilde{g}_n(x)\leq\tilde{g}_n(1/2)+2(2x-1)(\tilde{g}_n(3/4)-\tilde{g}_n(1/2))$, which implies that $\tilde{g}_n(x)\leq 3\max_{[1/4,3/4]}\{\tilde{g}_n\}$,
\item if $x\in[1/4,3/4]$, then $\tilde{g}_n(x)\leq\max_{[1/4,3/4]}\{\tilde{g}_n\}$.
\end{itemize}
Hence, for every $n\in\N$, we have $\lvert \tilde{g}_n\rvert_{L^{\infty}}\leq3\max_{[1/4,3/4]}\{\tilde{g}_n\}$. However, $\{\tilde{g}_n\}$ converges uniformly to $g_{\infty}$ on $[1/4,3/4]$, hence $\sup_{n\in\N}\{\max_{[1/4,3/4]}\{\tilde{g}_n\}\}<\infty$, which concludes the proof of $\Psi$'s continuity. Therefore, $\M_{0,N}(I)$ is homeomorphic to $\R_{>0}\times\{\Class^*/\{\pm1\}\}$, which is itself homeomorphic to $\R\times\{\Class^*/\{\pm1\}\}$.\\
\underline{Now, note that $\Class^*/\{\pm1\}$ is contractible.} Indeed, observe that the map $H\colon I\times\{\Class^*/\{\pm1\}\}\to\Class^*/\{\pm1\}$ defined by $H(t,[f])\coloneqq[t\tilde{1}+(1-t)f]$ is a retract by deformation of $\Class^*/{\pm1}$ onto $\{[\tilde{1}]\}$ (where $\tilde{1}$ is the function constantly equal to $1$). This concludes the proof of Proposition \ref{prop9.3}.
\end{proof}

\section{Moduli spaces in the 2-dimensional
case}\label{sec7}

In this section, we will describe the moduli spaces $\M_{0,2}(X)$, where $X$ has dimension $2$. We will start with the spaces whose splitting degree $k(X)$ is equal to $2$, namely the $2$-torus and the Klein bottle $\mathbb{K}^2$. We will then proceed with the spaces satisfying $k(X)=1$, namely the cylinder $\mathbb{S}^1\times I$ and the M\"obius band $\mathbb{M}^2$. Finally, we will study the case where $k(X)=0$, which corresponds to the $2$-sphere $\mathbb{S}^2$, the projective plane $\R\mathbb{P}^2$, and the closed $2$-disc $\mathbb{D}$.

\subsection{The 2-torus and the Klein bottle}

\subsubsection{The 2-torus}

\begin{Proposition}\label{prop: torus}The moduli space $\M_{0,2}(\mathbb{T}^2)$ is homeomorphic to $\R^4$; in particular, it is contractible.
\end{Proposition}

\begin{proof}According to Proposition \ref{proposition: the case of flat manifolds}, there is a homeomorphism:
\begin{equation}
\M_{0,2}(\mathbb{T}^2)\simeq\R\times\mathscr{M}_{\mathrm{flat}}(\mathbb{T}^2).
\end{equation}
Moreover, as a result of Section 2.1 of \cite{Perez_20}, $\mathscr{M}_{\mathrm{flat}}(\mathbb{T}^2)$ is homeomorphic to $\R^3$. Therefore, $\M_{0,2}(\mathbb{T}^2)$ is homeomorphic to $\R^4$.
\end{proof}

\subsubsection{The Klein bottle}

\begin{Proposition}\label{prop: klein}The moduli space $\M_{0,2}(\mathbb{K}^2)$ of $\RCD(0,2)$-structures on $\mathbb{K}^2$ is homeomorphic to $\R^3$; in particular, it is contractible.
\end{Proposition}

\begin{proof}We denote $\Gamma$ the Bieberbach subgroup of $\Iso(\R^2)$ generated by $a\coloneqq(I_2,e_1)$ and $b\coloneqq(\sigma,e_2)$, where $(e_1,e_2)$ is the canonical basis of $\R^2$, $I_2=\diag(1,1)$ is the identity matrix, and $\sigma\coloneqq\diag(-1,1)$. Let us recall that by definition $\mathbb{K}^2=\R^2/\Gamma$. Therefore, thanks to Proposition \ref{proposition: the case of flat manifolds}, there is a homeomorphism:
\begin{equation}\label{eq: klein}
\M_{0,2}(\mathbb{K}^2)\simeq\R\times[\Or_2(\R)\backslash\mathcal{C}_{\Gamma}]/\mathcal{N}_{\Gamma}.
\end{equation}
It is then readily checked that $
\mathrm{N}_{\Aff(\R^2)}(\Gamma)=\{(\diag(\epsilon_1,\epsilon_2),v),\epsilon_i\in\{\pm1\}, 2\braket{v,e_1}\in\Z\}$.
Therefore, we have $\mathcal{N}_{\Gamma}=\{\diag(\epsilon_1,\epsilon_2),\epsilon_i\in\{\pm1\}\}$ (see Definition \ref{Notation: crystallographic}). Thanks to Proposition 4.8 of \cite{Bettiol-Derdzinski-Piccione_18}, we have $\mathcal{C}_{\Gamma}=\Or_2(\R)\cdot\mathcal{Z}$, where $\mathcal{Z}$ is the centralizer of $H_{\Gamma}$ in $\GL_{2}(\R)$. In addition, $H_{\Gamma}$ is the subgroup of $\Or_2(\R)$ generated by $\sigma$; hence, it is easy to see that $\mathcal{Z}=\{\diag(a_1,a_2),a_i\in\R^*\}$. Now, observe that $\mathcal{N}_{\Gamma}$ acts trivially on $\Or_2(\R)\backslash\mathcal{C}_{\Gamma}$. Indeed, given $A\in\mathcal{C}_{\Gamma}$ and $B\in\mathcal{N}_{\Gamma}$, there exists $C\in\Or_2(\R)$ and $D\in\mathcal{Z}$ such that $A=CD$. In particular, we have $AB=CDB=CBD$ (since both $B$ and $D$ are diagonal matrices). Therefore, using the fact that $B\in\mathcal{N}_{\Gamma}\subset\Or_2(\R)$ and $C\in\Or_2(\R)$, we obtain $[D]=[CD]=[CBD]$ (where $[\cdot]$ denotes the class of a matrix in $\Or_2(\R)\backslash\mathcal{C}_{\Gamma}$). However, $[A]=[CD]$ and $[A]\cdot B=[CDB]=[CBD]$. Hence, $[A]=[A]\cdot B$, i.e. $\mathcal{N}_{\Gamma}$ acts trivially on $\Or_2(\R)\backslash\mathcal{C}_{\Gamma}$. Thus, thanks to \eqref{eq: klein}, we have a homeomorphism:
\begin{equation}\label{eq: klein 2}
\M_{0,2}(\mathbb{K}^2)\simeq\R\times\Or_2(\R)\backslash\mathcal{C}_{\Gamma}.
\end{equation}
Now, thanks to Corollary 4.9 of \cite{Bettiol-Derdzinski-Piccione_18}, $\Or_2(\R)\backslash\mathcal{C}_{\Gamma}$ is homeomorphic to the quotient space $\Or_2(\R)\cap\mathcal{Z}\backslash \mathcal{Z}$. Observe that the map $\diag(a,b)\in\mathcal{Z}\to (\lvert a \rvert,\lvert b \rvert)\in \R_{>0}\times\R_{>0}$ passes to the quotient, giving rise to a homeomorphism $[\diag(a,b)]\in\Or_2(\R)\cap\mathcal{Z}\backslash\mathcal{Z}\to(\lvert a \rvert,\lvert b \rvert)\in \R_{>0}\times\R_{>0}$. Hence, using \eqref{eq: klein 2}, we finally obtain $
\M_{0,2}(\mathbb{K}^2)\simeq\R^3$.
\end{proof}

\subsection{The M\"obius band and the cylinder}\label{section: N=2 k(X)=1}

\subsubsection{The M\"obius band}\label{sec: The Mobius band}

\begin{Proposition}\label{prop: mobius}The moduli space $\M_{0,2}(\mathbb{M}^2)$ of $\RCD(0,2)$-structures on $\mathbb{M}^2$ is homeomorphic to $\R^3$; in particular, it is contractible.
\end{Proposition}

\begin{proof}Assume that $(\mathbb{M}^2,\dist,\m)$ is an $\RCD(0,2)$-structure on $\mathbb{M}^2$ and let $(I\times\R,\tilde{\dist},\tilde{\m})\coloneqq p^*(\mathbb{M}^2,\dist,\m)$ be the associated lift (where $p\colon I\times\R\to\mathbb{M}^2$ is the quotient map). Observe that, following the proof of Theorem \ref{prop: main} (case $k(X)=1$), there exists a splitting $\phi\colon(I\times\R,\tilde{\dist},\tilde{\m})\to([0,r],\de,a\mathcal{L}^1)\times(\R,\de,\mathcal{L}^1)$, where $a,r\in\R_{>0}$. Moreover, following the same proof, there exist a generator $\gamma$ of $\overline{\pi}_1(X)$ and $b>0$ such that, for every $(\overline{x},t)\in [0,r]\times\R$, we have $\phi_*(\gamma)(\overline{x},t)=(r-\overline{x},t+b)$. In particular, we have $[\mathbb{M}^2,\dist,\m]=[\mathbb{M}^2,\dist_{r,b},a\mathcal{H}^2]$, where $(\mathbb{M}^2,\dist_{r,b})$ is the metric quotient of $(I\times\R,r\de\times b\de)$ by the action of $\Z$ on $I\times\R$ defined by $c\cdot(\overline{x},t)\coloneqq(s^c(\overline{x}),t+c)$, for every $c\in\Z$ (where $s\colon \overline{x}\in[0,1]\to 1-\overline{x}\in[0,1]$). Furthermore, it is easily seen from the definitions that $\mathcal{A}([\mathbb{M}^2,\dist,\m])=(\R,\de)/b\Z$ and $\mathcal{S}([\mathbb{M}^2,\dist,\m])=[I,r\de,a\mathcal{H}^1]$ (where the Albanese variety and the soul are defined respectively in \eqref{eq: albanese variety} and \eqref{eq: soul}). In particular, we have $a=\Mass(\mathcal{S}([\mathbb{M}^2,\dist,\m]))/\Diam(\mathcal{S}([\mathbb{M}^2,\dist,\m]))$, $r=\Diam(\mathcal{S}([\mathbb{M}^2,\dist,\m]))$, and $b=2\Diam(\mathcal{A}([\mathbb{M}^2,\dist,\m]))$. Therefore, the map $\Phi\colon a,b,r\in(\R_{>0})^3\to[\mathbb{M}^2,\dist_{r,b},a\mathcal{H}^2]\in\M_{0,2}(\mathbb{M}^2)$
is invertible and its inverse satisfies:
$$
\Phi^{-1}([\mathbb{M}^2,\dist,\m])=\Big(\frac{\Mass(\mathcal{S}([\mathbb{M}^2,\dist,\m]))}{\Diam(\mathcal{S}([\mathbb{M}^2,\dist,\m]))},2\Diam(\mathcal{A}([\mathbb{M}^2,\dist,\m])),\Diam(\mathcal{S}([\mathbb{M}^2,\dist,\m]))\Big).
$$
Now, observe that $\Phi$ and $\Phi^{-1}$ are continuous. Indeed, notice that, according to Theorem \ref{Th: Albanese and soul continuity}, $\Phi^{-1}$ is continuous. We then assume that $(a_k,b_k,r_k)\to(a_{\infty},b_{\infty},r_{\infty})$. It is then readily checked that the sequence $\{(I\times\R,r_k\de\times b_k\de,a_k\mathcal{H}^2,0)\}$ converges to $(I\times\R,r_{\infty}\de\times b_{\infty}\de,a_{\infty}\mathcal{H}^2,0)$ in the equivariant pmGH topology (see Definition 2.8 of \cite{Mondino-Navarro_22}) w.r.t. the action of $\Z$ on $I\times\R$ we introduced. Hence, thanks to Theorem A of \cite{Mondino-Navarro_22}, the sequence of associated quotients $\{(\mathbb{M}^2,\dist_{r_k,b_k},a_k\mathcal{H}^2)\}$ converges to $(\mathbb{M}^2,\dist_{r_{\infty},b_{\infty}},a_{\infty}\mathcal{H}^2)$ in the mGH topology. Thus, $\Phi$ is continuous, which concludes the proof.
\end{proof}

\subsubsection{The cylinder}

\begin{Proposition}\label{prop: cylinder}The moduli space $\M_{0,2}(\mathbb{S}^1\times I)$ of $\RCD(0,2)$-structures on $\mathbb{S}^1\times I$ is homeomorphic to $\R^3$; in particular, it is contractible.
\end{Proposition}

\begin{proof}Proceeding precisely as in section \ref{sec: The Mobius band}, it is readily checked that the map $\Phi\colon a,b,r\in(\R_{>0})^3\to[\mathbb{S}^1\times I,\dist_{r,b},a\mathcal{H}^2]\in\M_{0,2}(\mathbb{S}^1\times I)$ is a homeomorphism, where $d_{r,b}=b\dist_{\mathbb{S}^1}\times r\de$ and $\dist_{\mathbb{S}^1}$ is the length metric on the circle with perimeter $1$. Therefore, the result follows.
\end{proof}

\subsection{The 2-sphere, the projective plane, and the closed disc}\label{sec: N=2 k(X)=0}

In this section, we will compute the moduli spaces of $\RCD(0,2)$-structures on the $2$-sphere $\mathbb{S}^2$, the projective plane $\R\mathbb{P}^2$, and the closed disc $\mathbb{D}$. As we will see later, these moduli spaces are all homeomorphic to specific spaces of convex compact subsets of $\R^3$. We will start by introducing some notations of convex geometry. We will then prove realisation results for nonnegatively curved metrics on $\R\mathbb{P}^2$ and $\mathbb{D}$. We will then introduce convergence lemmas that will be fundamental to prove continuity statements. Finally, we will compute the aforementioned moduli spaces.

\subsubsection{Notations}

We start by introducing some notations from \cite{Belegradek_18}.

\begin{Notation}[Spaces of convex compact subsets of $\R^3$]\label{not: Spaces of convex compacta}
We denote $\mathcal{K}$ the set of all compact convex subsets of $\R^3$ and $\mathcal{K}^s$ the subset of $\mathcal{K}$ whose elements have their Steiner point at the origin (see section 4 of \cite{Belegradek_II_18} for some properties of the Steiner point).\\
Given $0\leq k\leq l\leq 3$ and $\mathcal{K}'\subset\mathcal{K}$, we denote:
\begin{itemize}
\item $\mathcal{K}'_{k\leq l}\coloneqq\{D\in\mathcal{K}',\dim(D)\in[k,l]\}$,
\item $\tilde{\mathcal{K}}'\coloneqq\{D\in\mathcal{K}',D=-D\}$.
\end{itemize}
Every subspace of $\mathcal{K}$ will be endowed with the Hausdorff distance $\dha$. Also, we will generically denote $B$, $K$, $L$ and $\{*\}$ elements of $\mathcal{K}$ with dimension $3$, $2$, $1$ and $0$.\end{Notation}

\begin{Remark}Observe that if $D\in\tilde{\mathcal{K}}$, then $s(D)=s(-D)=-s(D)$ (where $s(D)$ is the steiner point of $D$), i.e. $s(D)=0$. In particular, we have $\tilde{\mathcal{K}}\subset\mathcal{K}^s$.
\end{Remark}

The following maps will be important later when comparing the boundaries of two different convex bodies in $\mathcal{K}^s$.

\begin{Notation}[Central projection]\label{not: central projection}Let $n\geq1$ and assume that $D$ is an $n$-dimensional compact convex subset of $\R^n$ whose Steiner point is at the origin. Given $x\in\R^n\backslash\{0\}$, the open half line $\R_{>0}\cdot x$ intersects $\partial D$ in a single point, which we denote $p_{\partial D}^{\mathrm{c}}(x)$. The map $p_{\partial D}^{\mathrm{c}}\colon\R^n\backslash\{0\}\to\partial D$ is called the \emph{central projection on $\partial D$}.
\end{Notation}

A classical result is that the orthogonal projection on a closed convex subspace of a Hilbert space is well defined.

\begin{Notation}[Orthogonal projection]\label{not: orthogonal projection}Let $n\geq 1$ and assume that $D$ is a closed convex subset of $\R^n$. Given $x\in\R^n$, there exists a unique point $p_D(x)\in D$ such that $\de(x,p_D(x))=\de(x,D)$. The map $p_D\colon \R^n\to D$ is called the \emph{orthogonal projection on $D$}. Given $\alpha\in\R^n\backslash\{0\}$, we denote $p_{\alpha}\coloneqq p_{\R\alpha}$, and $p^{\perp}_{\alpha}\coloneqq p_{\{\alpha\}^{\perp}}$.
\end{Notation}

Let us now clarify what we mean when we speak about the boundary and the interior of a set.

\begin{Notation}[Boundary and interior]Let $M\subset \R^n$ ($n\geq 1$) be a topological submanifold with boundary. We denote $\partial M$ the boundary of $M$ and $\mathring{M}\coloneqq M\backslash\partial M$ the interior of $M$ (we will also write $\mathrm{int}(M)$).\end{Notation}
Every Lipschitz submanifold of $\R^n$ ($n\geq1$) admits two canonical metrics, namely the extrinsic and intrinsic metrics.

\begin{Notation}[Intrinsic and extrinsic metrics]
Given $n\in\N$ and given a connected Lipschitz submanifold $X\subset \R^n$ (possibly with boundary), we denote $\dist_X$ the \emph{intrinsic metric of $X$}. More precisely, given $x,y\in X$, we have $\dist_X(x,y)=\inf\{\mathcal{L}(\gamma)\}$, where the infimum is computed over the set of rectifiable curves in $X$ joining $x$ to $y$.\\
The \emph{extrinsic metric} on $X$ is simply the restriction of the Euclidean distance $\de$ to $X\subset \R^n$.\\
We will usually only write $X$ to speak about the metric space $(X,\dist_X)$, and we will specify when we endow $X$ with its extrinsic distance $\de$. For example, given $B\in\mathcal{K}^s$ of dimension $3$, will usually write $\partial B$ to speak about the metric space $(\partial B,\dist_{\partial B})$.
\end{Notation}

Let us introduce the double of a metric space. This notion will be crucial to realise metrics on the $2$-sphere, the projective plane and the disc.

\begin{Notation}[Double of a metric space]\label{not: double of K}
Given a topological manifold $X$ with boundary, we denote:
$$
\mathcal{D}X\coloneqq \sqcup_{\mathbf{i}=\mathbf{1},\mathbf{2}}\{\mathbf{i}\}\times X/\sim,
$$
where $(\mathbf{1},x)\sim(\mathbf{2},x)$, whenever $x\in\partial X$. We call $\mathcal{D}X$ the \emph{double of $X$} and let $q$ be the quotient map.\\
Given $\mathbf{i}\in\{\mathbf{1},\mathbf{2}\}$, we define $X^{\mathbf{i}}$ (resp. $\mathring{X}^{\mathbf{i}}$) as the image of $\{\mathbf{i}\}\times X$ (resp. $\{\mathbf{i}\}\times \mathring{X}$) under $q$. Since $\{\mathbf{1}\}\times\partial X$ and $\{\mathbf{2}\}\times \partial X$ are identified when passing to the quotient, we also denote $\partial X$ the image of these by $q$. Given $x\in X$, we will then write $x^{\mathbf{i}}\coloneqq q(\mathbf{i},x)\in X^{\mathbf{i}}$.\\
Given a length metric $\dist$ on $X$, there exists a unique length metric $\mathcal{D}\dist$ on $\mathcal{D}X$ whose restriction to $X^\mathbf{1}$ and $X^{\mathbf{2}}$ coincides with $\dist$. More precisely, given $x,y\in X$, we define $\mathcal{D}\dist(x^{\mathbf{1}},y^{\mathbf{1}})\coloneqq\dist(x,y)\eqqcolon\mathcal{D}\dist(x^{\mathbf{2}},y^{\mathbf{2}})$ and:
$$
\mathcal{D}\dist(x^{\mathbf{1}},y^{\mathbf{2}})\coloneqq\inf\{\dist(x,z)+\dist(z,y),z\in\partial X\}.
$$
We denote $\mathcal{D}(X,\dist)\coloneqq(\mathcal{D}X,\mathcal{D}\dist)$.\\
Very often, there will be no confusion about the metric $X$ is endowed with; therefore, we usually only write $\mathcal{D}X$ instead of $\mathcal{D}(X,\dist)$. For example, given $K\in\mathcal{K}$ of dimension $2$, we will write $\DK$ to speak about the metric space $\mathcal{D}(K,\de)$ (note that since $K$ is convex, the intrinsic metric $\dist_K$ coincides with the restriction of $\de$ to $K$).
\end{Notation}

The following maps will be relevant when studying double of metric spaces.

\begin{Notation}\label{not: inversion}Assume that $K$ is a $2$-dimensional compact convex subset of $\R^2$. We denote $s_K\colon\DK\to\DK$ the isometry defined by $s(x^{\mathbf{1}})\coloneqq x^{\mathbf{2}}$ and $s(x^{\mathbf{2}})\coloneqq x^{\mathbf{1}}$, for $x\in K$.
\end{Notation}

\begin{Notation}\label{not: double iso}Assume that $(X,\dist_X)$ and $(Y,\dist_Y)$ are metric spaces homeomorphic to topological manifolds with boundary and assume that $\phi\colon(X,\dist_X)\to(Y,\dist_Y)$ is an isometry (in particular $\phi(\partial X)=\partial Y$). We denote $\phi_{\mathcal{D}}\colon\mathcal{D}(X,\dist_X)\to\mathcal{D}(Y,\dist_Y)$ the isometry defined by $\phi_{\mathcal{D}}(x^{\mathbf{i}})\coloneqq(\phi(x))^{\mathbf{i}}$, for $x\in X$ and $\mathbf{i}\in\{\mathbf{1},\mathbf{2}\}$.
\end{Notation}

\subsubsection{Realisation of nonnegatively curved metrics}

In \cite{Belegradek_18}, Belegradek introduces a homeomorphism between the quotient space $\mathcal{K}^s_{2\leq3}/\Or_3(\R)$ and the moduli space $\mathscr{M}_{\mathrm{curv}\geq0}(\mathbb{S}^2)$ of nonnegatively curved metrics on $\mathbb{S}^2$. Let us first recall the correspondence.

\begin{Notation}\label{not: phi S2}Given $D\in\mathcal{K}^s$, we denote $\Phi_{\mathbb{S}^2}(D)\coloneqq \partial D$ if $\dim(D)=3$, $\Phi_{\mathbb{S}^2}(D)\coloneqq \mathcal{D}D$ if $\dim(D)=2$, and $\Phi_{\mathbb{S}^2}(D)\coloneqq D$ if $\dim(D)\in\{0,1\}$. In each case, $\Phi_{\mathbb{S}^2}(D)$ is endowed with its natural length metric.
\end{Notation}

Notice that given $D\in\Ks$, the isometry class $[\Phi_{\mathbb{S}^2}(D)]$ belongs to $\mathscr{M}_{\mathrm{curv}\geq0}(\mathbb{S}^2)$. Moreover, given $\phi\in\Or_3(\R)$, we have $[\Phi_{\mathbb{S}^2}(D)]=[\Phi_{\mathbb{S}^2}(\phi(D))]$. Therefore, there exists a unique map:
\begin{equation}\label{eq: homeo S2}
\Psi_{\mathbb{S}^2}\colon\mathcal{K}^s_{2\leq3}/\Or_3(\R)\to\mathscr{M}_{\mathrm{curv}\geq0}(\mathbb{S}^2)
\end{equation}
such that, for every $D\in\mathcal{K}^s_{2\leq3}$, we have $\Psi_{\mathbb{S}^2}([D])=[{\Phi}_{\mathbb{S}^2}(D)]$.\\

The following result is inspired by the realisation Theorem of Alexandrov (see Theorem 1 page 237 in \cite{Kutateladze_05}) and is proven by Belegradek in section 2 of \cite{Belegradek_18}.

\begin{Theorem}\label{th: homeo S2}The map $\Psi_{\mathbb{S}^2}\colon\mathcal{K}^s_{2\leq3}/\Or_3(\R)\to\mathscr{M}_{\mathrm{curv}\geq0}(\mathbb{S}^2)$ introduced in \eqref{eq: homeo S2} is a homeomorphism.
\end{Theorem}

We are now going to show that the equivariant nonnegatively curved metrics on $\mathbb{S}^2$ are in correspondence with symmetric convex compact subsets of $\R^3$ with dimension between $2$ and $3$. In what follows, by equivariant, we always mean equivariant with respect to the action of $\{\pm1\}$ on the relevant spaces.

\begin{Remark}Note that there might be various actions of $\{\pm1\}$ on our spaces. Given $B\in\tK$ of dimension $3$, we let $-1$ act as $-\id_{\R^3}$ on $\partial B$. Given $K\in\tK$ of dimension $2$, we let $-1$ act on $\DK$ in the following way: for $x\in K$, $-1\cdot x^{\mathbf{1}}\coloneqq(-x)^{\mathbf{2}}$ and $-1\cdot x^{\mathbf{2}}\coloneqq(-x)^{\mathbf{1}}$.
\end{Remark}

First, note that if $D\in\tK$, it is clear that there exists an equivariant nonnegatively curved metric $\dist$ on $\mathbb{S}^2$ such that $(\mathbb{S}^2,\dist)$ is equivariantly isometric to $\Phi_{\mathbb{S}^2}(D)$; hence, $[\Phi_{\mathbb{S}^2}(D)]\in\mathscr{M}^{\mathrm{eq}}_{\mathrm{curv\geq0}}(\mathbb{S}^2)$. In addition, if $\phi\in\Or_3(\R)$, then $\Phi_{\mathbb{S}^2}(D)$ and $\Phi_{\mathbb{S}^2}(\phi(D))$ are equivariantly isometric. Therefore, there exists a unique map:
\begin{equation}\label{eq: homeo S2 eq}
\Psi_{\mathbb{S}^2}^{\mathrm{eq}}\colon\tK/\Or_3(\R)\to\mathscr{M}^{\mathrm{eq}}_{\mathrm{curv\geq0}}(\mathbb{S}^2)
\end{equation}
such that, for every $D\in\tK$, we have $\Psi_{\mathbb{S}^2}^{\mathrm{eq}}([D])=[\Phi_{\mathbb{S}^2}(D)]$. We are going to prove the following realisation theorem for equivariant metrics on the $2$-sphere.

\begin{Proposition}\label{prop: realisation RP2}The map $\Psi_{\mathbb{S}^2}^{\mathrm{eq}}\colon\tK/\Or_3(\R)\to\mathscr{M}^{\mathrm{eq}}_{\mathrm{curv\geq0}}(\mathbb{S}^2)$ introduced in \eqref{eq: homeo S2 eq} is a $1$-$1$ correspondence.
\end{Proposition}

To prove Proposition \ref{prop: realisation RP2}, we will need the following Lemma.

\begin{Lemma}\label{lem: involutive isometry}Let $K\subset\R^2$ and $\Sigma\subset\R^2$ be $2$-dimensional compact convex subsets of $\R^2$. If $\phi\colon\mathcal{D}K\to\mathcal{D}\Sigma$ is an isometry, then $\phi(\partial K)=\partial\Sigma$.
\end{Lemma}

\begin{proof}Looking for a contradiction, let us assume that $\partial K\cap\phi^{-1}(\mathring{\Sigma}^{\mathbf{2}})\neq\varnothing$.\\
First of all, observe that if $x,y\in\partial K\cap\phi^{-1}(\mathring{\Sigma}^{\mathbf{2}})$, then there exists a unique unit speed shortest path $[\phi(x)\phi(y)]$ from $\phi(x)$ to $\phi(y)$. In particular, there exists a unique unit speed shortest path $[xy]$ from $x$ to $y$. In addition, since $s_K$ is an isometry of $\DK$ (see Notation \ref{not: inversion}), then $s_K([xy])=[s_K(x)s_K(y)]$ is the unique unit speed shortest path from $s_K(x)$ to $s_K(y)$. However, since $x,y\in\partial K$, then we have $s_K(x)=x$ and $s_K(y)=y$, which implies $s_K([xy])=[xy]$. Therefore, for every $z\in[xy]$, we have $z\in\Fix(s_K)=\partial K$. Hence, we necessarily have $[xy]\subset\partial K$. Moreover, observe that $[\phi(x)\phi(y)]\subset\mathring{\Sigma}^{\mathbf{2}}$ and that $\phi([xy])=[\phi(x)\phi(y)]$; therefore, $[xy]\subset\partial K\cap\phi^{-1}(\mathring{\Sigma}^{\mathbf{2}})$.\\
After that, observe that if $x,y,z \in\partial K\cap\phi^{-1}(\mathring{\Sigma}^{\mathbf{2}})$, then $x,y,z$ are necessarily aligned. Indeed, if not, then the triangle $\Delta(x,y,z)$ is homeomorphic to $\mathbb{S}^1$ and is a subset of $\partial K\cap\phi^{-1}(\mathring{\Sigma}^{\mathbf{2}})$ (using the first part of the proof), which implies $\Delta(x,y,z)=\partial K$. In particular, thanks to the first observation, every pair of points on $\partial K$ can be joined by a unique unit speed shortest path, which is a contradiction.\\
Using the two observations above, there exist $x\neq y\in\partial K$ such that $(xy)=\partial K\cap\phi^{-1}(\mathring{\Sigma}^{\mathbf{2}})$, where $(xy)$ is the interior of $[xy]$ seen as a subspace of $\partial K\subset\R^2$. Note that we necessarily have:
\begin{itemize}
\item[(i)] $[xy]\subset\partial K$,
\item[(ii)] $\phi(x),\phi(y)\in\partial \Sigma$.
\end{itemize}
In particular, there exists a unique unit speed shortest path from $x$ to $y$ in $\mathcal{D}K$ (using (i)). Therefore, there is a unique unit speed shortest path from $\phi(x)$ to $\phi(y)$, which implies that $\phi([xy])=[\phi(x)\phi(y)]\subset\partial \Sigma$ (using (ii)). However, let us recall that by assumption $(xy)=\partial K\cap\phi^{-1}(\mathring{\Sigma}^{\mathbf{2}})$. Hence, $\phi((xy))\subset\partial \Sigma\cap\mathring{\Sigma}^{\mathbf{2}}=\varnothing$, which is the contradiction we were seeking.\\
The same argument leads to $\partial K\cap\phi^{-1}(\mathring{\Sigma}^\mathbf{1})=\varnothing$.\\
Now, we have shown that $\phi(\partial K)\subset\partial\Sigma$. However, the same argument applied to $\phi^{-1}$ leads to $\phi^{-1}(\partial\Sigma)\subset\partial K$; therefore, $\partial\Sigma\subset\phi(\partial K)$, which concludes the proof.
\end{proof}

Lemma \ref{lem: involutive isometry} implies the following Proposition.

\begin{Proposition}\label{prop: involutive isometry}Let $K\subset\R^2$ and $\Sigma\subset\R^2$ be $2$-dimensional compact convex subsets of $\R^2$ with Steiner point at the origin and let $\phi\colon\mathcal{D}K\to\mathcal{D}\Sigma$ be an isometry. There exists $\mu\in\Or_2(\R)$ such that $\mu(K)=\Sigma$ and $\phi=\mu_{\mathcal{D}}$ if $\phi(K^{\mathbf{2}})=\Sigma^{\mathbf{2}}$ (resp. $s_{\Sigma}\circ\phi=\phi\circ s_K=\mu_{\mathcal{D}}$ if $\phi(K^{\mathbf{2}})=\Sigma^\mathbf{1}$), where we introduced $\mu_{\mathcal{D}}$ (resp. $s_{\Sigma}$ and $s_K$) in Notation \ref{not: double iso} (resp. Notation \ref{not: inversion}).
\end{Proposition}

\begin{proof}First of all, observe that thanks to Lemma \ref{lem: involutive isometry}, we have $\phi(\partial K)=\partial\Sigma$. In particular, $\phi$ maps the connected components $\mathring{K}^\mathbf{1}$ and $\mathring{K}^{\mathbf{2}}$ to $\mathring{\Sigma}^\mathbf{1}\sqcup\mathring{\Sigma}^{\mathbf{2}}$. Hence, composing $\phi$ with $s_{\Sigma}$ if necessary, we can assume that $\phi(K^{\mathbf{i}})=\Sigma^{\mathbf{i}}$ ($\mathbf{i}\in\{\mathbf{1},\mathbf{2}\}$). In particular, there exists isometries $\mu\colon K\to \Sigma$ and $\nu\colon K\to \Sigma$ such that $\mu_{\lvert\partial K}=\nu_{\lvert\partial K}$ and such that, for every $x\in K$, we have $\phi(x^{\mathbf{2}})=\mu(x)^{\mathbf{2}}$ and $\phi(x^\mathbf{1})=\nu(x)^\mathbf{1}$. Thanks to Theorem 2.2 of \cite{Alestalo-Trotsenko-Vaisala_01}, we can assume that $\mu$ and $\nu$ are isometries of $\R^2$. However $\Span(\partial K)=\R^2$; therefore, since $\mu$ and $\nu$ coincide on $\partial K$, we necessarily have $\mu=\nu$. In addition, $K$ and $\Sigma$ have their Steiner point at the origin and $\mu(K)=\Sigma$; thus, we necessarily have $\mu(0)=0$, i.e. $\mu\in\Or_2(\R)$. Hence, we can conclude that $\phi=\mu_{\mathcal{D}}$
\end{proof}

We are now able to prove Proposition \ref{prop: realisation RP2}.

\begin{proof}[Proof of Proposition \ref{prop: realisation RP2}]First of all, assume that $D_1,D_2\in\tK$ satisfy $\Psi_{\mathbb{S}^2}^{\mathrm{eq}}([D_1])=\Psi_{\mathbb{S}^2}^{\mathrm{eq}}([D_2])$. By definition, $\Phi_{\mathbb{S}^2}(D_1)$ is equivariantly isometric to $\Phi_{\mathbb{S}^2}(D_2)$. In particular $\Psi_{\mathbb{S}^2}([D_1])=\Psi_{\mathbb{S}^2}([D_2])$. Hence, thanks to Theorem \ref{th: homeo S2}, we have $[D_1]=[D_2]$. Thus, $\Psi_{\mathbb{S}^2}^{\mathrm{eq}}$ is injective.\\
Let us now show that $\Psi_{\mathbb{S}^2}^{\mathrm{eq}}$ is surjective. Let $\dist$ be an equivariant nonnegatively curved metric on $\mathbb{S}^2$ and let us show that there exists $D\in\tK$ such that $(\mathbb{S}^2,\dist)$ is equivariantly isometric to $\Phi_{\mathbb{S}^2}(D)$. Thanks to Theorem \ref{th: homeo S2}, either there exists $B\in\Ks$ of dimension $3$ such that $(\mathbb{S}^2,\dist)$ is isometric to $\partial B$, or there exists $K\in\Ks$ of dimension $2$ such that $(\mathbb{S}^2,\dist)$ is isometric to $\DK$.\\
Let us first assume that we have an isometry $\phi\colon (\mathbb{S}^2,\dist)\to\partial B$. We define $f\colon x\in\partial B\to \phi(-\phi^{-1}(x))\in\partial B$. It is sufficient to prove that $f$ coincides with $-\id_{\R^3}$. To do so, note that $f$ is a self-isometry of $\partial B$. Therefore, thanks to Theorem 5.2.1 of \cite{Burago-Zalgaller_92}, $f$ is also a self-isometry of $(\partial B,\de)$. Thus, thanks to Theorem 2.2 of \cite{Alestalo-Trotsenko-Vaisala_01}, we can extend $f$ into an isometry of $(\R^3,\de)$ (which we also denote $f$). Observe that $f$ is in particular an affine transformation; thus, $f(B)=f(\{\mathrm{Conv}(\partial B)\})=\mathrm{Conv}\{f(\partial B)\}=\mathrm{Conv}\{\partial B\}=B$. Therefore, we have $f(0)=f(s(B))=s(f(B))=s(B)=0$ (where $s(B)$ is the Steiner point of $B$), i.e. $f\in\Or_3(\R)$. Now, note that by definition $f$ is involutive on $\partial B$. Moreover, $\partial B$ spans $\R^3$. Hence, $f$ is an orthogonal involution of $\R^3$ and can be diagonalized with eigenvalues $\pm 1$. However, by definition, $f$ has no fixed point on $\partial B$, so, $1$ cannot be an eigenvalue of $f$. Hence, $f=-\id_{\R^3}$.\\
Now, let us assume that we have an isometry $\phi\colon(\mathbb{S}^2,\dist)\to\DK$. Observe that without loss of generality, we may assume that $K\subset \R^2\times\{0\}$. As above, we introduce $f\colon x\in\DK\to \phi(-\phi^{-1}(x))\in \DK$, which is an involutive isometry without any fixed points. Let us prove that $f$ coincides with the action of $-1$ on $\DK$. Applying Lemma \ref{lem: involutive isometry}, we have either $f(K^{\mathbf{2}})\subset K^{\mathbf{2}}$ or $f(K^{\mathbf{2}})\subset K^{\mathbf{1}}$. If $f(K^{\mathbf{2}})\subset K^{\mathbf{2}}$, then Brouwer's fixed point Theorem implies that $f$ has a fixed point on $K^{\mathbf{2}}$ (which is not possible by definition of $f$). Hence, we necessarily have $f(K^{\mathbf{2}})\subset K^{\mathbf{1}}$. Note that, thanks to Proposition \ref{prop: involutive isometry}, there exists an isometry $\mu\in\Or_2(\R)$ such that $\mu(K)=K$ and $s_{K}\circ f=\mu_{\mathcal{D}}$ (see notations \ref{not: inversion} and \ref{not: double iso}). In addition, let us recall that $f^{2}=\id_{\DK}$. Thus, using $\mu_{\mathcal{D}}\circ s_{K}=s_{K}\circ \mu_{\mathcal{D}}$ and $s_{K}^{2}=\id_{\DK}$, we obtain that $\mu$ is involutive on $K$. However, $\Span(K)=\R^2$; hence, $\mu^{2}=\id_{\R^2}$. In particular, $\mu$ is diagonalisable on $\R^2$ with eigenvalues $\pm1$. However, if $1$ is an eigenvalue, then $f$ admits a fixed point on the boundary of $K$, which can't happen. Therefore, $\mu=-\id_{\R^2}$ which implies that $K=-K$ and that $f$ coincides with the action of $-1$ on $\DK$.
\end{proof}

Now, we are going to focus on how to realise nonnegatively curved metrics on $\mathbb{D}$ using convex compact subsets of $\R^3$. But first, we need to introduce some notations.

\begin{Notation}\label{not: moduli space of the closed disc}Given $\alpha\in\mathbb{S}^2$, we denote $H_{\alpha}^{-}\coloneqq\{\braket{\alpha,\cdot}\leq0\}$ and $H_{\alpha}^{+}\coloneqq\{\braket{\alpha,\cdot}\geq0\}$ respectively the lower half-space and upper half-space induced by $\alpha$, and $H_{\alpha}\coloneqq\{\alpha\}^{\perp}$. We write $r_{\alpha}$ for the reflection w.r.t. $H_{\alpha}$. We then denote $\mathcal{K}^{\alpha}\coloneqq\{D\in\mathcal{K}^s,r_{\alpha}(D)=D\}$ and:
$$
\mathscr{K}\coloneqq\bigcup_{\alpha\in\mathbb{S}^2}\mathcal{K}^{\alpha}\times\{\alpha\}\subset\mathcal{K}^s\times\mathbb{S}^2,
$$
and $\mathscr{K}_{2\leq3}\coloneqq\mathscr{K}\cap\Ks\times\mathbb{S}^2$. The last two spaces introduced are endowed with the direct topology.
\end{Notation}

\begin{Remark}\label{rem: condition on alpha}If $(K,\alpha)\in\mathscr{K}$ satisfies $\dim(K)=2$, then we have either $\alpha\in\Span(K)$ or $\alpha\in\Span(K)^{\perp}$. Indeed, let us assume that $\alpha\notin\Span(K)$ and let us show that, in that case, $r_{\alpha}$ coincides with $\id$ on $\Span(K)$. If $r_{\alpha}$ does not coincide with $\id$ on $\Span(K)$, then there exists $x\in\Span(K)\backslash\{0\}$ such that $r_{\alpha}(x)=-x$ (using $r_{\alpha}(K)=K$). Moreover, since $\alpha\notin\Span(K)$, then $\Span(x,\alpha)$ has dimension $2$. However, $r_{\alpha}$ coincides with $-\id$ on $\Span(x,\alpha)$, which contradicts the fact that $\Dim(\Ker(r_{\alpha}+\id))=1$. Therefore, $r_{\alpha}$ necessarily coincides with $\id$ on $\Span(K)$. Hence, $\Ker(r_{\alpha}-\id)=\Span(K)=\{\alpha\}^{\perp}$, i.e. $\alpha\in\Span(K)^{\perp}$.\\
Proceeding with the same idea, we can show that if $(L,\alpha)\in\mathscr{K}$ satisfies $\dim(L)=1$, then either $\alpha\in\Span(L)$ or $\alpha\perp \Span(L)$.
\end{Remark}

The following subsets associated with the double of a plane region will be crucial to obtain nonnegatively curved metrics on $\mathbb{D}$.

\begin{Notation}\label{not: DKalpha}Assume that $(K,\alpha)\in\mathscr{K}$ such that $\dim(K)=2$ and $\alpha\in\Span(K)$. We denote $\DK_{\alpha}^{\pm}\coloneqq\cup_{{\mathbf{i}}={\mathbf{1}}}^{\mathbf{2}}(K\cap H_{\alpha}^{\pm})^{\mathbf{i}}\subset\DK$, $\mathcal{D}\mathring{K}_{\alpha}^{\pm}\coloneqq\cup_{{\mathbf{i}}={\mathbf{1}}}^{\mathbf{2}}(K\cap \mathring{H}_{\alpha}^{\pm})^{\mathbf{i}}\subset\DK$, and $\DK_{\alpha}\coloneqq\cup_{{\mathbf{i}}={\mathbf{1}}}^{\mathbf{2}}(K\cap H_{\alpha})^{\mathbf{i}}=\DK_{\alpha}^{+}\cap\DK_{\alpha}^{-}\subset\DK$.
Observe that all of the sets above are convex subsets of $\DK$. In particular, $\DK_{\alpha}^{+}$ is isometric to $\mathbb{D}$ endowed with a nonnegatively curved metric.\end{Notation}

It is always possible to symmetrise a convex compact subsets of $\R^3$ w.r.t. a specific direction.

\begin{Notation}\label{not: symmetrisation}Given $\alpha\in\mathbb{S}^2$ and $D\in\mathcal{K}^s$, we denote $D^{\alpha}\coloneqq (D+r_{\alpha}(D))/2$, where $+$ is the Minkowski sum.
\end{Notation}

We now introduce the map that will lead to the correspondence between $\mathscr{M}_{\mathrm{curv}\geq0}(\mathbb{D})$ and a particular space of convex compact subsets of $\R^3$.

\begin{Notation}\label{def: disc correspondence}Given $(D,\alpha)\in\mathscr{K}$ (see Notation \ref{not: moduli space of the closed disc}), we denote:
\begin{itemize}
\item[(i)] $\Phi_{\mathbb{D}}(D,\alpha)\coloneqq\partial D\cap H_{\alpha}^+$ if $\dim(D)=3$,
\item[(ii)] $\Phi_{\mathbb{D}}(D,\alpha)\coloneqq D$ if $\alpha\in\Span(D)^{\perp}$,
\item[(iii)] $\Phi_{\mathbb{D}}(D,\alpha)\coloneqq\{\mathcal{D}D\}_{\alpha}^+$ if $\dim(D)=2$ and $\alpha\in\Span(D)$ (see Notation \ref{not: DKalpha}),
\item[(iv)] $\Phi_{\mathbb{D}}(D,\alpha)\coloneqq D\cap H_{\alpha}^+$ if $\dim(D)=1$ and $\alpha\in\Span(D)$.
\end{itemize}
Note that in any case we have $\Phi_{\mathbb{D}}(D,\alpha)\subset\Phi_{\mathbb{S}^2}(D)$ (see Notation \ref{not: phi S2}). In each case, $\Phi_{\mathbb{D}}(D,\alpha)$ is endowed with its natural length metric.
\end{Notation}

First, observe that if $(D,\alpha)\in\mathscr{K}_{2\leq3}$ (see Notation \ref{not: moduli space of the closed disc}), it is clear that there exists a nonnegatively curved metric $\dist$ on $\mathbb{D}$ such that $(\mathbb{D},\dist)$ is isometric to $\Phi_{\mathbb{D}}(D,\alpha)$; hence, $[\Phi_{\mathbb{D}}(D,\alpha)]\in\mathscr{M}_{\mathrm{curv\geq0}}(\mathbb{D})$. Note that if $\phi\in\Or_3(\R)$, then $\Phi_{\mathbb{D}}(D,\alpha)$ and $\Phi_{\mathbb{D}}(\phi(D),\phi(\alpha))$ are isometric. Therefore, there exists a unique map:
\begin{equation}\label{eq: homeo disc}
\Psi_{\mathbb{D}}\colon\mathscr{K}_{2\leq3}/\Or_3(\R)\to\mathscr{M}_{\mathrm{curv\geq0}}(\mathbb{D})
\end{equation}
such that, for every $(D,\alpha)\in\mathscr{K}_{2\leq3}$, we have $\Psi_{\mathbb{D}}([D,\alpha])=[\Phi_{\mathbb{D}}(D,\alpha)]$. Our next goal is to prove the following proposition.

\begin{Proposition}\label{prop: realisation disc}The map $\Psi_{\mathbb{D}}\colon\mathscr{K}_{2\leq3}/\Or_3(\R)\to\mathscr{M}_{\mathrm{curv\geq0}}(\mathbb{D})$ introduced in \eqref{eq: homeo disc} is a $1$-$1$ correspondence.
\end{Proposition}

\begin{proof}Let us first prove that $\Psi_{\mathbb{D}}$ is surjective. Assume that $\dist$ is a nonnegatively curved metric on $\mathbb{D}$. Thanks to Perelman doubling Theorem (see section 13.3 of \cite{Burago-Gromov-Perelman_92}), the double metric space $\mathcal{D}(\mathbb{D},\dist)$ (see Notation \ref{not: double of K}) is also an Alexandrov space with nonnegative curvature. In addition, $\mathcal{D}\mathbb{D}$ is homeomorphic to $\mathbb{S}^2$. Therefore, thanks to Theorem 1.1 of \cite{Belegradek_18}, either there exists $B\in\Ks$ of dimension $3$ such that $\mathcal{D}(\mathbb{D},\dist)$ is isometric to $\partial B$, or there exists $K\in\Ks$ of dimension $2$ such that $\mathcal{D}(\mathbb{D},\dist)$ is isometric to $\DK$. Before considering each case, let us denote $\phi\colon\mathcal{D}\mathbb{D}\to\mathcal{D}\mathbb{D}$ the map defined by $\phi(x^{\mathbf{1}})\coloneqq x^{\mathbf{2}}$ and $\phi(x^{\mathbf{2}})\coloneqq x^{\mathbf{1}}$, for every $x\in\mathbb{D}$. Observe that $\phi$ is an involutive isometry of $\mathcal{D}(\mathbb{D},\dist)$, whose set of fixed points $\Fix(\phi)$ is equal to $\partial\mathbb{D}=\mathbb{S}^1$.\\
Let us consider the case where there exists an isometry $f\colon\mathcal{D}(\mathbb{D},\dist)\to\partial B$. Note that the map $\psi\coloneqq f\circ\phi\circ f^{-1}$ is an involutive isometry of $\partial B$ whose set of fixed points $\Fix(\psi)$ is homeomorphic to $\mathbb{S}^1$. Arguing as for the $3$-dimensional case in the proof of Proposition \ref{prop: realisation RP2}, we can assume that $\psi$ belong to $\Or_3(\R)$, satisfies $\psi^{2}=\id_{\R^3}$, and $\psi(B)=B$. In particular, $\psi$ is diagonalisable with eigenvalues $\pm1$. Observe that $1$ necessarily has multiplicity $2$ in order for $\Fix(\psi)$ to be homeomorphic to $\mathbb{S}^1$; hence, there exists $\alpha\in\mathbb{S}^2$ such that $\psi=r_{\alpha}$. In particular, $r_{\alpha}(B)=B$, i.e. $(B,\alpha)\in\mathscr{K}_{2\leq3}$ (see Notation \ref{not: moduli space of the closed disc}). Now, notice that $f(\Fix(\phi))=\Fix(\psi)$, i.e. $f(\partial\mathbb{D})=\partial B\cap H_{\alpha}$. Hence, $f$ maps the disjoint union $\mathring{\mathbb{D}}^{\mathbf{1}}\sqcup\mathring{\mathbb{D}}^{\mathbf{2}}$ onto $\{\partial B\cap\mathring{H}_{\alpha}^{-}\}\sqcup\{\partial B\cap\mathring{H}_{\alpha}^{+}\}$. In particular, changing $\alpha$ into $-\alpha$ if necessary, we have $f(\mathring{\mathbb{D}}^{\mathbf{2}})=\partial B\cap\mathring{H}_{\alpha}^{+}$. Therefore, $f$ is an isometry from $(\mathbb{D},\dist)$ to $\Phi_{\mathbb{D}}(B,\alpha)$.\\
Now, we assume that there exists an isometry $f\colon\mathcal{D}(\mathbb{D},\dist)\to\DK$ and, without loss of generality, that $K\subset \R^2\times\{0\}$. As before, $\psi\coloneqq f\circ\phi\circ f^{-1}$ is an involutive isometry of $\DK$ such that $\mathrm{Fix}(\psi)$ is homeomorphic to $\mathbb{S}^1$. Thanks to Lemma \ref{lem: involutive isometry}, we either have $\psi(K^{\mathbf{2}})\subset K^{\mathbf{2}}$ or $\psi(K^{\mathbf{2}})\subset K^\mathbf{1}$.\\
If $\psi(K^{\mathbf{2}})\subset K^{\mathbf{2}}$, then by Proposition \ref{prop: involutive isometry} there exists $\mu\in\Or_2(\R)$ such that $\mu(K)=K$ and $\psi=\mu_{\mathcal{D}}$ (see Notation \ref{not: double iso}). Since $\psi$ is involutive and $\Span(K)=\R^2$, we have $\mu^{2}=\id_{\R^2}$ which implies that $\mu$ is diagonalisable with eigenvalues $\pm1$. We then note that $1$ necessarily has multiplicity $1$, otherwise $\Fix(\psi)$ is not homeomorphic to $\mathbb{S}^1$. Hence, there exists $\alpha\in\Span(K)$ such that $\mu=r_{\alpha}$; in particular, $(K,\alpha)\in\mathscr{K}_{2\leq3}$. Observe that $f(\mathrm{Fix}(\phi))=f(\partial \mathbb{D})=\mathrm{Fix}(\psi)=\DK_{\alpha}$ (see Notation \ref{not: DKalpha}). Therefore, $f$ maps the disjoint union $\mathring{\mathbb{D}}^\mathbf{1}\sqcup\mathring{\mathbb{D}}^{\mathbf{2}}$ onto $\mathcal{D}\mathring{K}_{\alpha}^{-}\sqcup\mathcal{D}\mathring{K}_{\alpha}^{+}$. In particular, changing $\alpha$ into $-\alpha$ if necessary, we may assume that $f(\mathring{D}^{\mathbf{2}})=\mathcal{D}\mathring{K}_{\alpha}^{+}$. Therefore, $f$ is an isometry from $(\mathbb{D},\dist)$ to $\mathcal{D}{K}_{\alpha}^{+}=\Phi_{\mathbb{D}}(K,\alpha)$.\\ 
If $\psi(K^{\mathbf{2}})\subset K^\mathbf{1}$, then we can use Proposition \ref{prop: involutive isometry} and proceed as above to get $\mu\in\Or_2(\R)$, such that $\mu(K)=K$, $\mu^{2}=\id_{\R^2}$, and $s_K\circ\psi=\mu_{\mathcal{D}}$ (see notations \ref{not: inversion} and \ref{not: double iso}). As before, $\mu$ is diagonalisable with eigenvalues $\pm1$. Observe that $1$ necessarily has multiplicity $2$ for $\Fix(\psi)$ to be homeomorphic to $\mathbb{S}^1$; therefore, $\mu=\id_{\R^2}$. We then note that $f$ maps $\Fix(\phi)=\partial \mathbb{D}$ onto $\Fix(\psi)=\partial K$. Hence, composing $f$ with $s_K$ if necessary, we can assume that $f(\mathring{D}^{\mathbf{2}})=\mathring{K}^{\mathbf{2}}$. Therefore, fixing any $\alpha\in\Span(K)^{\perp}$, $f$ induces an isometry between $(\mathbb{D},\dist)$ and $K=\Phi_{\mathbb{D}}(K,\alpha)$.\\
Now let us prove that $\Psi_{\mathbb{D}}$ is injective. Let $(D_i,\alpha_i)\in\mathscr{K}_{2\leq3}$ ($i\in\{1,2\}$) and assume that there exists an isometry $\phi \colon\phi_{\mathbb{D}}(D_1,\alpha_1)\to\Phi_{\mathbb{D}}(D_2,\alpha_2)$. We need to prove that there exists $\psi\in\Or_3(\R)$ such that $D_2=\psi(D_1)$ and $\alpha_2=\psi(\alpha_1)$. First of all, observe that $\phi$ induces an isometry $\phi_{\mathcal{D}}\colon\mathcal{D}\{\Phi_{\mathbb{D}}(D_1,\alpha_1)\}\to\mathcal{D}\{\Phi_{\mathbb{D}}(D_2,\alpha_2)\}$. Moreover, there exists an isometry $\nu_i\colon\mathcal{D}\{\Phi_{\mathbb{D}}(D_i,\alpha_i)\}\to\Phi_{\mathbb{S}^2}(D_i)$ (this is readily checked via a case by case study). In particular, thanks to Theorem \ref{th: homeo S2}, we necessarily have $\dim(D_1)=\dim(D_2)$.\\
If $\dim(D_1)=\dim(D_2)=3$, then $\psi\coloneqq\nu_2\circ\phi_{\mathcal{D}}\circ\nu_1^{-1}\colon\partial D_1\to\partial D_2$ is an isometry. Arguing as for the $3$-dimensional case in the proof of Proposition \ref{prop: realisation RP2}, we can assume that $\psi\in\Or_3(\R)$ satisfies $\psi(D_1)=D_2$. Moreover, we can easily chose $\nu_i$ ($i\in\{1,2\}$) so that $\psi(\partial D_1\cap H_{\alpha_1})=\partial D_2\cap H_{\alpha_2}$, which implies that $\psi(H_{\alpha_1})=H_{\alpha_2}$. Hence, composing $\psi$ with $r_{\alpha_2}$ if necessary, we can conclude that $\psi(\alpha_1)=\alpha_2$.\\
Let us assume that $\dim(D_1)=\dim(D_2)=2$. If $\alpha_i\in\Span(D_i)^{\perp}$ ($i\in\{1,2\}$), then $\phi\colon D_1\to D_2$ is an isometry. Thanks to Theorem 2.2 of \cite{Alestalo-Trotsenko-Vaisala_01} and since $s(D_1)=s(D_2)=0$, we can assume that $\phi\in\Or_3(\R)$. In particular, $\phi(\Span(D_1)^{\perp})=\Span(D_2)^{\perp}$. Therefore, composing $\phi$ with $r_{\alpha_2}$ if necessary, we obtain $\phi(D_1)=D_2$ and $\phi(\alpha_1)=\alpha_2$. If $\alpha_1\in\Span(D_1)^{\perp}$ and $\alpha_2\in\Span(D_2)$, then $D_1$ is isometric to $\{\mathcal{D}{D_2}\}_{\alpha_2}^+$. In particular, we necessarily have $\phi(\partial D_1)=\partial \{\mathcal{D}{D_2}\}_{\alpha_2}^+=\{\mathcal{D}{D_2}\}_{\alpha_2}$. However, $\{\mathcal{D}{D_2}\}_{\alpha_2}$ is a convex subset of $\{\mathcal{D}{D_2}\}_{\alpha_2}$ and $\partial D_1$ is not a convex subset of $D_1$; thus, that case cannot happen. The case where $\alpha_i\in\Span(D_i)$ ($i\in\{1,2\}$) can be treated in the same way as the case $\dim(D_1)=\dim(D_2)=3$.
\end{proof}

\subsubsection{Approximation Lemmas}\label{sec: Approximation Lemmas}

In this section, we introduce approximation lemmas. The goal here is to define GH approximations between spaces with various dimensions. This will be crucial later when we will prove that $\Psi_{\mathbb{S}^2}^{\mathrm{eq}}$ (see \eqref{eq: homeo S2 eq}) and $\Psi_{\mathbb{D}}$ (see \eqref{eq: homeo disc}) are homeomorphisms.\\

First of all, let us recall the following result (see the proof of Lemma 10.2.7 of \cite{Burago-Ivanov_01}).

\begin{Lemma}[$3$ to $3$]\label{lem: 3 to 3}Let $
B,B'\in\mathcal{K}^s$ such that $\dim(B)=\dim(B')=3$ and assume that there exists $\epsilon\in(0,1)$ such that $(1-\epsilon)B\subset B'\subset (1+\epsilon) B$ and $(1-\epsilon)B'\subset B\subset (1+\epsilon) B'$. If we denote $f$ and $g$ respectively the restriction of $p^{\mathrm{c}}_{\partial B}$ to $\partial B'$ and the restriction of $p^{\mathrm{c}}_{\partial B'}$ to $\partial B$ (see Notation \ref{not: central projection}), then $(f,g)$ is a GH $\nu$-approximation (see Remark \ref{remark: equivariant GH topology}) between $\partial B'$ and $\partial B$, where $\nu\coloneqq6(\Diam(B)+\Diam(B'))\epsilon$. Moreover, if $B,B'\in\tilde{\mathcal{K}}$ (see Notation \ref{not: Spaces of convex compacta}), then $(f,g,\id)$ is an equivariant GH $\nu$-approximation between $(\partial B',\{\pm1\})$ and $(\partial B,\{\pm1\})$.
\end{Lemma}

\begin{Remark}The result is a bit different than what appears in the proof of Lemma 10.2.7 of \cite{Burago-Ivanov_01}; we just used the fact that $\Diam(\partial B)\leq\pi\Diam(B)$ (which also appears in the proof mentioned above).
\end{Remark}

Let us introduce the following notation, which relates $3$-dimensional and $2$-dimensional convex spaces.

\begin{Notation}\label{not: 3 to 2}Let $B\in\mathcal{K}^s$ such that $\dim(B)=3$, let $v\in \mathbb{S}^2$, and denote $K\coloneqq p^{\perp}_{v}(B)$ (see Notation \ref{not: orthogonal projection}). There exist two functions $\phi^{\mathbf{1}}\colon K\to \R$ and $\phi^{\mathbf{2}}\colon K\to \R$ that are respectively convex and concave such that, for every $x\in K$, we have:
$$
(x+\R v)\cap B=[x+\phi^{\mathbf{1}}(x) v,x+\phi^{\mathbf{2}}(x) v].
$$
Moreover, $\partial B$ can be partitioned as $\partial B=\partial B^{\mathbf{1}}\sqcup \partial B^{\mathbf{2}}\sqcup \partial B^{\mathbf{L}}$, where $\partial B^{\mathbf{i}}\coloneqq\{x+\phi^{\mathbf{i}}(x)v,x\in\mathring{K}\}$ ($\mathbf{i}\in\{\mathbf{1},\mathbf{2}\}$) and $\partial B^{\mathbf{L}}\coloneqq\{y, y\in[x+\phi^{\mathbf{1}}(x) v,x+\phi^{\mathbf{2}}(x) v],x\in\partial K\}$. This allows us to introduce:
$$
f_{B,v}\colon \partial B\to \mathcal{D}K,
$$
the map defined by $f_{B,v}(x+\phi^{\mathbf{i}}(x)v)\coloneqq x^{\mathbf{i}}$ (for $x\in\mathring{K}$ and $\mathbf{i}\in\{\mathbf{1},\mathbf{2}\}$) and $f_{B,v}(y)\coloneqq x$ (for every $x\in\partial K$ and $y\in[x+\phi^{\mathbf{1}}(x) v,x+\phi^{\mathbf{2}}(x) v]$). Finally, we introduce:
$$
g_{B,v}\colon\mathcal{D}K\to\partial B,
$$
the map defined by $g_{B,v}(x^{\mathbf{i}})\coloneqq x+\phi^{\mathbf{i}}(x)v$ (for $x\in\mathring{K}$ and $\mathbf{i}\in\{\mathbf{1},\mathbf{2}\}$) and $g_{B,v}(x)\coloneqq x+[(\phi^{\mathbf{1}}(x)+\phi^{\mathbf{2}}(x))/2] v$ (for $x\in\partial K$).
\end{Notation}

The next lemma introduces approximations between a $3$-dimensional convex space and a $2$-dimensional one.

\begin{Lemma}[$3$ to $2$]\label{lem: 3 to 2}Let $B\in\mathcal{K}^s_{2\leq3}$ such that $\dim(B)=3$, let $v\in \mathbb{S}^2$, and denote $K\coloneqq p^{\perp}_{v}(B)$. If we denote $f\coloneqq f_{B,v}$, $g\coloneqq g_{B,v}$ (see Notation \ref{not: 3 to 2}), and $\epsilon\coloneqq\sup_{x\in B}\{\de(x,p^{\perp}_{v}(x))\}$, then $(f,g)$ is a GH $10\epsilon$-approximation between $\partial B$ and $\DK$. Moreover, if $B=-B$, then $(f,g,\id)$ is an equivariant GH $10\epsilon$-approximation between $(\partial B,\{\pm1\})$ and $(\DK,\{\pm1\})$.
\end{Lemma}

\begin{proof}First of all, assume that $x,y\in\mathring{K}$. Given $t\in[0,1]$, set $\alpha(t)\coloneqq tx+(1-t)y$ and $\gamma(t)\coloneqq\alpha(t)+\phi_{\mathbf{2}}(\alpha(t))v$. Observe that $\gamma$ is a curve with values in $\partial B$ such that $\gamma(0)=g(x^{\mathbf{2}})$ and $\gamma(1)=g(y^{\mathbf{2}})$. In particular, $\dist_{\partial B}(g(x^{\mathbf{2}}),g(y^{\mathbf{2}}))\leq\mathcal{L}(\gamma)\leq\mathcal{L}(\alpha)+\mathcal{L}(\beta)$, where $\beta(t)\coloneqq\phi_{\mathbf{2}}(\alpha(t))$. However, $\mathcal{L}(\alpha)=\de(x,y)=\dist_{\DK}(x^{\mathbf{2}},y^{\mathbf{2}})$. In addition, $\beta$ is a concave function from $[0,1]$ to $[-\epsilon,\epsilon]$ (using the fact that $\phi_{\mathbf{2}}$ is concave and $B\subset K+[-\epsilon,\epsilon]v$). Hence, there exists $t_0\in[0,1]$ such that $\beta$ is increasing on $[0,t_0]$ and non-increasing on $[t_0,1]$. Therefore, $\mathcal{L}(\beta)=\int_{0}^1\lvert\beta'\rvert=\int_{0}^{t_0}\beta'-\int_{t_0}^1\beta'=2\beta(t_0)-\beta(1)-\beta(0)$. In particular, we have $\mathcal{L}(\beta)\leq 4\epsilon$ which implies that:
\begin{equation}\label{eq: approx 3 to 2_1}
\dist_{\partial B}(g(x^{\mathbf{2}}),g(y^{\mathbf{2}}))-\dist_{\DK}(x^{\mathbf{2}},y^{\mathbf{2}})\leq4\epsilon.
\end{equation}
Conversely, let us assume that $\gamma$ is a geodesic between $g(x^{\mathbf{2}})$ and $g(y^{\mathbf{2}})$ on $\partial B$. We then note that $p^{\perp}_{v}(\gamma)$ has shorter length and joins $x$ to $y$, thus:
\begin{equation}\label{eq: approx 3 to 2_2}
\dist_{\DK}(x^{\mathbf{2}},y^{\mathbf{2}})=\de(x,y)\leq\mathcal{L}(p^{\perp}_{v}(\gamma))\leq\mathcal{L}(\gamma)=\dist_{\partial B}(g(x^{\mathbf{2}}),g(y^{\mathbf{2}})).
\end{equation}
Hence, as a result of \eqref{eq: approx 3 to 2_1} and \eqref{eq: approx 3 to 2_2}, we obtain:
\begin{equation}\label{equation: revisions1}
\lvert\dist_{\partial B}(g(x^{\mathbf{2}}),g(y^{\mathbf{2}})-\dist_{\DK}(x^{\mathbf{2}},y^{\mathbf{2}})\rvert \leq4\epsilon.
\end{equation}
The same argument works with $x^{\mathbf{2}}$ and $y^{\mathbf{2}}$ replaced by $x^{\mathbf{1}}$ and $y^{\mathbf{1}}$, respectively.\\
Now, assume that $x\in\mathring{K}$, $z\in\partial K$, and $\mathbf{i}\in\{\mathbf{1},\mathbf{2}\}$. Observe that, proceeding as we did for \eqref{eq: approx 3 to 2_2}, we have $\dist_{\DK}(x^{\mathbf{i}},z) \leq \dist_{\partial B}(g(x^{\mathbf{i}}),g(z))$. Proceeding as we did for \eqref{eq: approx 3 to 2_1}, we have the following inequality:
\begin{equation*}
\dist_{\partial B}(g(x^{\mathbf{i}}),z+\phi_{\mathbf{i}}(z))\leq\dist_{\DK}(x^{\mathbf{i}},z)+4\epsilon.
\end{equation*}
However, note that $\dist_{\partial B}(z+\phi_{\mathbf{i}}(z),g(z))=\lvert \phi_{\mathbf{2}}(z)-\phi_{\mathbf{1}}(z)\rvert/2\leq \epsilon$ (since $B\subset K+[-\epsilon,\epsilon]v$). Therefore, we obtain $\dist_{\partial B}(g(x^{\mathbf{i}}),g(z))\leq\dist_{\DK}(x^{\mathbf{i}},z)+4\epsilon+\dist_{\partial B}(z+\phi_{\mathbf{i}}(z),g(z))\leq \dist_{\DK}(x^{\mathbf{i}},z)+5\epsilon$. Hence, we have the following inequality:
\begin{equation}\label{equation: revisions2}
\lvert \dist_{\partial B}(g(x^{\mathbf{i}}),g(z))-\dist_{\DK}(x^{\mathbf{i}},z)\rvert \leq 5\epsilon.\\
\end{equation}
The argument in the previous paragraph implies that, for every $z,z'\in\partial K$, we obtain the following inequality:
\begin{equation}\label{equation: revisions3}
\lvert \dist_{\partial B}(g(z),g(z'))-\dist_{\DK}(z,z')\rvert \leq 6\epsilon.\\
\end{equation}
Finally, assume that $x,y\in\mathring{K}$ and let $z\in\partial K$ such that $\dist_{\DK}(x^{\mathbf{1}},y^{\mathbf{2}})=\de(x,z)+\de(z,y)$. Applying the argument we used to obtain \eqref{eq: approx 3 to 2_1} and recalling that $B$ is a subset of $K+[-\epsilon,\epsilon]v$, we obtain the following inequality:
\begin{equation*}
\begin{split}
\dist_{\partial B}(g(x^{\mathbf{1}}),g(y^{\mathbf{2}}))&\leq\dist_{\partial B}(g(x^{\mathbf{1}}),z+\phi_{\mathbf{1}}(z))+\dist_{\partial B}(z+\phi_{\mathbf{1}}(z),z+\phi_{\mathbf{2}}(z))+\dist_{\partial B}(z+\phi_{\mathbf{2}}(z),g(y^{\mathbf{2}}))\\
&\leq 4\epsilon+\de(x,z)+2\epsilon+4\epsilon+\de(z,y)\\
&=10\epsilon+\dist_{\DK}(x^{\mathbf{1}},y^{\mathbf{2}}).
\end{split}
\end{equation*}
Let us then fix a geodesic $\gamma\colon[0,1]\to\partial B$ from $g(x^{\mathbf{1}})$ to $g(y^{\mathbf{2}})$. Observe that $\partial B^{\mathbf{2}}$ is an open subset of $\partial B$ and its boundary is a subset of $\partial B^{\mathbf{L}}$ (see Notation \ref{not: 3 to 2}). Therefore, there exists $0<t_2<1$ such that $\gamma_{\lvert (t_2,1]}\subset \partial B^{\mathbf{2}}$ and $\gamma(t_2)\in\partial B^{\mathbf{L}}$. In particular, we have the following inequality:
\begin{equation*}
\begin{split}
\dist_{\DK}(x^{\mathbf{1}},y^{\mathbf{2}})&\leq \de(x,p_v^{\perp}(\gamma(t_2)))+\de(p_v^{\perp}(\gamma(t_2)),y)\\
&=\de(p_v^{\perp}(\gamma(0)),p_v^{\perp}(\gamma(t_2)))+\de(p_v^{\perp}(\gamma(t_2)),p_v^{\perp}(\gamma(1)))\\
&\leq \mathcal{L}(\gamma_{\lvert [0,t_2]})+\mathcal{L}(\gamma_{\lvert [t_2,1]})=\mathcal{L}(\gamma)=\dist_{\partial B}(g(x^{\mathbf{1}}),g(y^{\mathbf{2}})).
\end{split}
\end{equation*}
Thus, we obtain the following inequality:
\begin{equation}\label{equation: revisions4}
\lvert\dist_{\partial B}(g(x^{\mathbf{1}}),g(y^{\mathbf{2}})-\dist_{\DK}(x^{\mathbf{1}},y^{\mathbf{2}})\rvert \leq10\epsilon.
\end{equation}
As a result of \eqref{equation: revisions1}, \eqref{equation: revisions2},\eqref{equation: revisions3}, and \eqref{equation: revisions4}, we obtain $\Dis(g)\leq10\epsilon$ (see Notation \ref{not: distortion}).\\
It is then easy to see that, for every $x\in\partial B$, we have $\dist_{\partial B}(g\circ f(x),x)\leq 2\epsilon$ and that $f\circ g=\id_{\DK}$. Finally, the argument used to estimate the distortion of $g$ can be used to show that $\Dis(f)\leq10\epsilon$.\\
Now, notice that if $B=-B$, then, for every $x\in K$, we have $\phi_{\mathbf{1}}(-x)=-\phi_{\mathbf{2}}(x)$ and $\phi_{\mathbf{2}}(-x)=-\phi_{\mathbf{1}}(x)$. Therefore, both $f$ and $g$ are equivariant.
\end{proof}

The next lemma compares a $3$-dimensional convex space with its projection on a line.

\begin{Lemma}[$3$ to $1$]\label{lem: 3 to 1}Let $B\in\mathcal{K}^s$ such that $\dim(B)=3$, let $v\in \mathbb{S}^2$, and denote $L\coloneqq p_{v}(B)$ (see Notation \ref{not: orthogonal projection}). If we denote $f\colon\partial B\to L$ the restriction of $p_v$ to $\partial B$, and $\epsilon\coloneqq\sup_{x\in B}\{\de(x,p_{v}(x))\}$, then we have $\Dis(f)\leq (8+\pi)\epsilon$ (see Notation \ref{not: distortion}).
\end{Lemma}

\begin{proof}First, denoting $\mathcal{B}\coloneqq B_{\epsilon}(0)\cap\{v\}^{\perp}$, observe that $B\subset L+\mathcal{B}\eqqcolon C$. Let $x,y\in\partial B$ and observe that $
\dist_{L}(f(x),f(y))=\de(p_{v}(x),p_{v}(y))\leq\de(x,y)\leq\dist_{\partial B}(x,y)$. Hence:
\begin{equation}\label{eq: approx 3 to 1_1}
\forall x,y\in\partial B, 0\leq\dist_{\partial B}(x,y)-\dist_{L}(f(x),f(y)).
\end{equation}
Now, it is not hard to check that $p_{B}\colon \partial C\to \partial B$ is surjective (see Notation \ref{not: orthogonal projection}). Let $x_i\in\partial B$ and let $x_i'\in\partial C$ such that $x_i=p_{B}(x_i')$ ($i\in\{1,2\}$). Thanks to the Busemann--Feller Lemma (see the proof of Lemma 10.2.7 in \cite{Burago-Ivanov_01}), we have the following inequality:
\begin{equation}\label{eq: revisions2}
\dist_{\partial B}(x_1,x_2)-\dist_{L}(f(x_1),f(x_2))\leq\dist_{\partial C}(x_1',x_2')-\de(p_{v}(x_1),p_{v}(x_2)).
\end{equation}
It is not trivial to find an upper bound for the term $\dist_{\partial C}(x_1',x_2')$. However, observe that the set $L+\partial \mathcal{B}$ (endowed with its intrinsic metric) is a metric product and is a subset of $\partial C$. Moreover, note that there exists $x_i''\in L+\partial \mathcal{B}$ such that $\dist_{\partial C}(x_i',x_i'')=\de(x_i',x_i'')\leq\epsilon$ ($i\in\{1,2\}$). In particular, we have the following inequality:
\begin{equation}\label{eq: revisions1}
\dist_{\partial C}(x_1',x_2')\leq 2\epsilon + \dist_{\partial C}(x_1'',x_2'')\leq 2\epsilon+\dist_{L+\partial \mathcal{B}}(x_1'',x_2'').
\end{equation}
Let $\gamma$ be a geodesic from $x_1''$ to $x_2''$ in ${L+\partial \mathcal{B}}$. Since $L+\partial \mathcal{B}$ is a metric product, there exists a geodesic $\gamma_L$ in $L$ and a geodesic $\gamma_{\partial\mathcal{B}}$ in $\partial \mathcal{B}$ such that $\gamma=\gamma_L+\gamma_{\partial \mathcal{B}}$. Therefore, we have the following inequality:
\begin{align*}
\dist_{L+\partial \mathcal{B}}(x_1'',x_2'')=\mathcal{L}(\gamma)&\leq\mathcal{L}(\gamma_{\partial \mathcal{B}})+\mathcal{L}(\gamma_{L})\\&\leq \Diam(\partial \mathcal{B})+\de(p_v(x_1''),p_v(x_2''))\\
&\leq \pi\epsilon+\de(p_v(x_1''),p_v(x_2'')).
\end{align*}
In particular, according to \eqref{eq: revisions2} and \eqref{eq: revisions1}, we have the following inequality:
\begin{align*}
\dist_{\partial B}(x_1,x_2)-\dist_{L}(f(x_1),f(x_2))&\leq(2+\pi)\epsilon+\de(p_{v}(x_1''),p_{v}(x_2''))-\de(p_{v}(x_1),p_{v}(x_2))\\
&\leq(2+\pi)\epsilon+\sum_{i=1,2}\de(x_i,x_i')+\de(x_i',x_i'')\\
&\leq (4+\pi)\epsilon+\sum_{i=1,2}\de(x_i,x_i'),
\end{align*}
where we used the fact that $p_{v}$ is $1$-Lipschitz for the second inequality. Moreover, note that $\de(x_i,x_i')=\de(x_i',B)$ ($i\in\{1,2\}$). In addition, writing $x_i'=t_i+v_i$ for some $t_i\in L$ and $v_i\in \mathcal{B}$ ($i\in\{1,2\}$), it is clear that $\{t_i+\mathcal{B}\}\cap B\neq\varnothing$. In particular, there exists $w_i\in \mathcal B$ ($i\in\{1,2\}$) such that $t_i+w_i\in B$. Therefore, $\de(x_i', B)\leq \de (t_i+v_i,t_i+w_i)\leq\Diam(\mathcal{B})=2\epsilon$. Hence, we also have shown that:
\begin{equation}\label{eq: approx 3 to 1_2}
\forall x,y\in\partial B,\dist_{\partial B}(x,y)-\dist_{L}(f(x),f(y))\leq (8+\pi)\epsilon.
\end{equation}
Thanks to \eqref{eq: approx 3 to 1_1} and \eqref{eq: approx 3 to 1_2}, we can conclude that $\Dis(f)\leq(8+\pi)\epsilon$.
\end{proof}

We now introduce a way to compare doubles of plane convex regions.

\begin{Lemma}[$2$ to $2$]\label{lem: 2 to 2}Let $
K,K'\in\mathcal{K}^s$ such that $\dim(K)=\dim(K')=2$ and assume that there exists $\epsilon\in(0,1)$ such that $(1-\epsilon)K\subset K'\subset (1+\epsilon) K$ and $(1-\epsilon)K'\subset K\subset (1+\epsilon) K'$ (in particular, $K$ and $K'$ are coplanar). We denote $f\colon\mathcal{D}K\to\mathcal{D}K'$ the map defined by $f(x^{\mathbf{i}})\coloneqq[(1-\epsilon)x]^{\mathbf{i}}$ (for $x\in\mathring{K}$ and $\mathbf{i}\in\{\mathbf{1},\mathbf{2}\}$) and $f(x)\coloneqq p_{\partial K'}^{\mathrm{c}}(x)$ (for $x\in\partial K$). We define $g\colon \mathcal{D}K'\to\mathcal{D}K$ in the same way by exchanging the roles of $K$ and $K'$. The pair $(f,g)$ is a GH $\nu$-approximation between $\DK$ and $\DK'$, where $\nu\coloneqq4(\Diam(K)+\Diam(K'))\epsilon$. Moreover, if $K,K'\in \tilde{\mathcal{K}}$, then $(f,g,\id)$ is an equivariant GH $\nu$-approximation between $(\DK,\{\pm1\})$ and $(\DK',\{\pm1\})$.
\end{Lemma}

\begin{proof}First, assume that $x,y\in\mathrm{int}(K)$, and observe that:
$$
\lvert \dist_{\mathcal{D}K'}(f(x^{\mathbf{2}}),f(y^{\mathbf{2}}))-\dist_{\mathcal{D}K}(x^{\mathbf{2}},y^{\mathbf{2}})\rvert=\epsilon\de(x,y).
$$
We then assume that $y\in\partial K$ and observe that $p_{\partial K'}^{\mathrm{c}}(y)=\lambda y$ for some $\lambda>0$. However, $(1-\epsilon)y\in K'$; hence, $\lambda\geq(1-\epsilon)$. Observe that $(1+\epsilon)^{-1}p_{\partial K'}^{\mathrm{c}}(y)\in K$. In particular, we have $q_{K}((1+\epsilon)^{-1}p_{\partial K'}^{\mathrm{c}}(y))\leq1$ (where $q_{K}$ is the Minkowski gauge associated to $K$). However, since $y\in\partial K$, then $q_{K}(y)=1$. Hence, using $q_{K}((1+\epsilon)^{-1}p_{\partial K'}^{\mathrm{c}}(y))=(1+\epsilon)^{-1}\lambda q_{K}(y)=(1+\epsilon)^{-1}\lambda\leq 1$, we have $\lambda\leq1+\epsilon$. In particular, we have $\lvert 1-\lambda\rvert\leq\epsilon$. Therefore, given $x\in\mathrm{int}(K)$ we have:
\begin{align*}
\lvert \dist_{\mathcal{D}K'}(f(x^{\mathbf{2}}),f(y))-\dist_{\mathcal{D}K}(x^{\mathbf{2}},y)\rvert&=\lvert \de((1-\epsilon)x,p_{\partial K'}^{\mathrm{c}}(y))-\de(x,y)\rvert\\
&\leq \de((1-\epsilon)x,x)+\de(\lambda y,y)\\
&\leq \lvert 1-\lambda\rvert \lvert y\rvert+\epsilon\Diam(K)\\
&\leq2\epsilon\Diam(K).
\end{align*}
Proceeding the same way, we also have $\lvert \dist_{\mathcal{D}K'}(f(x^{\mathbf{1}}),f(y))-\dist_{\mathcal{D}K}(x^{\mathbf{1}},y)\rvert\leq2\epsilon\Diam(K)$. Finally, given $x,y\in\mathrm{int}(K)$, there exists $z\in\partial K$ such that $\dist_{\mathcal{D}K}(x^{\mathbf{2}},y^{\mathbf{1}})=\de(x,z)+\de(y,z)$. Hence, we have:
\begin{align*}
\dist_{\mathcal{D}K'}(f(x^{\mathbf{2}}),f(y^{\mathbf{1}}))&\leq\dist_{\mathcal{D}K'}(f(x^{\mathbf{2}}),f(z))+\dist_{\mathcal{D}K'}(f(z),f(y^{\mathbf{1}}))\\
&\leq 4\epsilon\Diam(K)+\de(x,z)+\de(z,y)\\
&=4\epsilon\Diam(K)+\dist_{\mathcal{D}K}(x^{\mathbf{2}},y^{\mathbf{1}}).
\end{align*}
Moreover, there exists $z'\in\partial K'$ such that $\dist_{\mathcal{D}K'}(f(x^{\mathbf{2}}),f(y^{\mathbf{1}}))=\dist_{\mathcal{D}K'}(f(x^{\mathbf{2}}),z')+\dist_{\mathcal{D}K'}(z',f(y^{\mathbf{1}}))$. In addition, denoting $z''\coloneqq p_{\partial K}^{\mathrm{c}}(z')$, we have $z'=f(z'')$. Therefore, we obtain:
\begin{align*}
\dist_{\mathcal{D}K'}(f(x^{\mathbf{2}}),f(y^{\mathbf{1}}))&=\dist_{\mathcal{D}K'}(f(x^{\mathbf{2}}),f(z''))+\dist_{\mathcal{D}K'}(f(z''),f(y^{\mathbf{1}}))\\
&\geq\dist_{\mathcal{D} K}(x^{\mathbf{2}},z'')+\dist_{\mathcal{D}K}(z'',y^{\mathbf{1}})-4\epsilon\Diam(K)\\
&\geq \dist_{\mathcal{D}K}(x^{\mathbf{2}},y^{\mathbf{1}})-4\epsilon\Diam(K).
\end{align*}
Hence, we have shown that $\Dis(f)\leq4\epsilon\Diam(K)$. We show in the same way that $\Dis(g)\leq4\Diam(K')\epsilon$.\\
Now, observe that given $x\in K$ and $\mathbf{i}\in\{\mathbf{1},\mathbf{2}\}$, we have $\dist(x,g\circ f (x))\leq(1-(1-\epsilon)^2)\lvert x\rvert\leq2\Diam(K)\epsilon$. The same holds replacing $f$ by $g$ and exchanging the roles of $K$ and $K'$. Hence, $(f,g)$ is a GH $\nu$-approximation between $\DK$ and $\DK'$.\\
Finally, it is clear that $f$ and $g$ are equivariant whenever $K,K'\in\tilde{\mathcal{K}}_{2\leq2}$.
\end{proof}

The following result shows how to compare the double of a plane convex region with its projection onto a line.

\begin{Lemma}[$2$ to $1$]\label{lem: 2 to 1}Let $K\in\mathcal{K}^s$ such that $\dim(K)=2$, let $v\in\Span(K)\backslash\{0\}$, and let $L\coloneqq p_v(K)$. If we denote $f\colon \mathcal{D}K\to L$ the map defined by $f(x^{\mathbf{i}})\coloneqq p_v(x)$ (for $x\in K$ and $\mathbf{i}\in\{\mathbf{1},\mathbf{2}\}$), then we have $\Dis(f)\leq4\epsilon$, where $\epsilon\coloneqq\sup_{x\in K}\{\de(x,p_v(x))\}$.
\end{Lemma}

\begin{proof}Note first that if $x,y\in K$ and if $\mathbf{i}\in\{\mathbf{1},\mathbf{2}\}$, then:
\begin{align*}
\lvert \dist_{L}(f(x^{\mathbf{i}}),f(y^{\mathbf{i}}))-\dist_{\mathcal{D}K}(x^{\mathbf{i}},y^{\mathbf{i}})\rvert &=\lvert \de(p_{v}(x),p_{v}(y))-\de(x,y)\rvert\\
&\leq\de(x,p_{v}(x))+\de(y,p_{v}(y))\\
&\leq 2\epsilon.
\end{align*}
Now, assume that $x,y\in K$ and let $z\in\partial K$ such that $\dist_{\mathcal{D}K}(x^{\mathbf{2}},y^{\mathbf{1}})=\de(x,z)+\de(y,z)$. Note that we have the following inequalities:
\begin{align*}
\dist_{L}(f(x^{\mathbf{2}}),f(y^{\mathbf{1}}))&=\de(p_{v}(x),p_{v}(y))\\
&\leq \de(x,y)\\
&\leq\de(x,z)+\de(z,y)=\dist_{\mathcal{D}K}(x^{\mathbf{2}},y^{\mathbf{1}}).
\end{align*}
After that, let $\beta\in\Span(K)$ such that $\{v,\beta\}$ is an orthonormal basis of $\Span(K)$ (we can assume without loss of generality that $v$ is unitary). Let $s\in[-\epsilon,\epsilon]$ be chosen such that $z'\coloneqq (p_{v}(x)+p_{v}(y))/2+s\beta\in\partial K$. Note that we have:
\begin{align*}
\de(x,z')&\leq\de(x,p_{v}(x))+\de(z',p_{v}(z'))+\de(p_{v}(x),p_{v}(z'))\\
&\leq 2\epsilon+\de(p_{v}(x),p_{v}(z')),
\end{align*}
and, proceeding the same way, we have $\de(y,z')\leq2\epsilon+\de(p_{v}(z'),p_{v}(y))$. Hence, we have:
\begin{align*}
\dist_{\mathcal{D}K}(x^{\mathbf{2}},y^{\mathbf{1}})&\leq \de(x,z')+\de(y,z')\\
&\leq 4\epsilon+\de(p_{v}(x),p_{v}(z'))+\de(p_{v}(z'),p_{v}(y))\\
&\leq4\epsilon+\de(p_{v}(x),p_{v}(y))=4\epsilon+\dist_{L}(f(x^{\mathbf{2}}),f(y^{\mathbf{1}})),
\end{align*}
where we used the fact that $p_{v}(z')\in[p_{v}(x),p_{v}(y)]$. Therefore, we can conclude that $\Dis(f)\leq4\epsilon$.
\end{proof}

The next lemma is trivial but will be needed for completeness.

\begin{Lemma}[1 to 1]\label{lem: 1 to 1}Let $
L,L'\in\mathcal{K}^s$ such that $\dim(L)=\dim(L')=1$ and assume that there exists $\epsilon\in(0,1)$ such that $(1-\epsilon)L\subset L'\subset (1+\epsilon) L$ and $(1-\epsilon)L'\subset L\subset (1+\epsilon) L'$ (in particular, $L$ and $L'$ are collinear). We denote $f\colon L\to L'$ the map defined by $f(x)\coloneqq(1-\epsilon)x$. We define $g\colon L'\to L$ in the same way by exchanging the roles of $L$ and $L'$. The pair $(f,g)$ is a GH $\nu$-approximation between $L$ and $L'$, where $\nu\coloneqq2(\Diam(L)+\Diam(L'))\epsilon$.
\end{Lemma}

The following result treats the case where a convex compact subset of $\R^3$ collapses to a point.

\begin{Lemma}[Collapsing case]\label{lem: collapsing case} If $D\in\mathcal{K}^s$, then $\Diam(\Phi_{\mathbb{S}^2}(D))\leq\pi\Diam(D)$ (see Notation \ref{not: phi S2}).
\end{Lemma}

\begin{proof}First of all, note that the result is trivial if $\dim(D)\in\{0,1\}$. If $\dim(D)=2$, it is clear from the definition of the double of a metric space that $\Diam(\mathcal{D}D)\leq2\Diam(D)$. Finally, if $\dim(D)=3$, then we have $\Diam(\partial D)\leq\pi\Diam(D)$ (see the proof of Lemma 10.2.7 in \cite{Burago-Ivanov_01}).
\end{proof}

Given a sequence $D_n\to D_{\infty}$ in $\mathcal{K}^s$ with $\dim(D_n)=2$ (for $n\in\N\cup\{\infty\}$), there is not necessarily any $\epsilon_n\to0$ such that $(1-\epsilon_n)D_{\infty}\subset D_n\subset (1+\epsilon_n)D_{\infty}$. Indeed, this would hold only if we had $D_n\subset\Span(D_{\infty})$ when $n$ is large enough. The following lemma is going to help us fix this issue.

\begin{Lemma}\label{lem: orthogonal}If that $D_n\to D_\infty$ in $\mathcal{K}^s$ such that $\dim(D_{\infty})\leq\dim(D_n)$ (for every $n\in\N$), then there exists $\phi_n\to\id_{\R^3}$ in $\Or_3(\R)$ such that $\phi_n^{-1}(\Span(D_{\infty}))\subset\Span(D_{n})$. In addition, if $(D_n,\alpha_n)\to(D_{\infty},\alpha_{\infty})$ in $\mathscr{K}$ (see Notation \ref{not: moduli space of the closed disc}), we can also ask $\{\phi_n\}$ to satisfy $\phi_n(\alpha_n)=\alpha_{\infty}$.
\end{Lemma}

\begin{proof}First of all, we assume that $1\leq k\coloneqq\dim(D_{\infty})$ (the case $k=0$ being trivial). We then let $\{w_i\}_{i=1}^3$ be an oriented orthonormal basis of $\R^3$ such that $\Span(D_{\infty})=\Span(\{w_i\}_{i=1}^k)$. Let $r>0$ such that $B_r(0)\cap\Span(D_{\infty})\subset D_{{\infty}}$ (such an $r>0$ always exists since the Steiner point of $D_{\infty}$ is at the origin and belong to the relative interior of $D_{\infty}$). For every $n\in\N$ and $i\in\N\cap[1,k]$, there exists $u_i^n\in D_n$ such that $\de(u_i^n,rw_i)\leq\epsilon_n\coloneqq\dha(D_n,D_{\infty})\to0$. We can then apply the Gram--Schmidt orthonormalisation process to the family $\{u_i^n\}_{i=1}^k$ and get $\{v_i^n\}_{i=1}^k$ such that $\Span(\{v_i^n\}_{i=1}^k)\subset\Span(D_n)$ and, for $i\in\N\cap[1,k]$, $v_i^n\to w_i$.\\
Let us construct $\{v_i^n\}_{k<i}$ such that $\{v_i^n\}_{i=1}^3$ is an orthonormal basis of $\R^3$ and for $k<i$ we have $v_i^n\to w_i$. If $k=3$ we are done already. If $k=2$, we can just define $v_3^n\coloneqq v_1^n\wedge v_2^n$. Let us now assume that $k=1$. In that case, whenever $n$ is large enough, $p_{v_1^n}^{\perp}\colon\{v_1\}^{\perp}\to\{v_1^n\}^{\perp}$ is an isomorphism. We can then define $u_i^n\coloneqq p_{v_1^n}^{\perp}(w_i)$ for $i\in\{2,3\}$. Applying the Gram--Schmidt orthonormalisation process to the family $\{u_2^n,u_3^n\}$ gives rise to a family $\{v_2^n,v_3^n\}$ satisfying the desired properties.\\
To conclude the first part of the proof, let $\phi_n\in\Or_3(\R)$ such that $\phi_n(v_i^n)=w_i$ ($i\in\{1,2,3\}$) and observe that $\phi_n$ satisfies the desired properties by construction.\\
Now, let us assume that $(D_n,\alpha_n)\to(D_{\infty},\alpha_{\infty})$ in $\mathscr{K}$. It is readily checked, proceeding case by case (and remembering that $(D_n,\alpha_n)\in\mathscr{K}$ implies either $\alpha_n\in\Span(D_n)$ or $\alpha_n\perp\Span(D_n)$), that we can construct $\{v_i^n\}$ and $\{w_i\}$ so that $\alpha_{\infty}=w_i$ for some $i\in\{1,2,3\}$, and $v_i^n=\alpha_n$ for every $n\in\N$. This concludes the proof.
\end{proof}

\subsubsection{Moduli spaces of nonnegatively curved metrics}

We have seen in Theorem \ref{th: homeo S2} that $\mathscr{M}_{\mathrm{curv}\geq0}(\mathbb{S}^2)$ is homeomorphic to $\Ks/\Or_3(\R)$. We are now going to prove results in the same spirit for $\mathscr{M}_{\mathrm{curv}\geq0}^{\mathrm{eq}}(\mathbb{S}^2)$ and $\mathscr{M}_{\mathrm{curv}\geq0}(\mathbb{D})$.

\begin{Proposition}\label{prop: RP2 continuity}The map $\Psi_{\mathbb{S}^2}^{\mathrm{eq}}\colon\tK/\Or_3(\R)\to\mathscr{M}^{\mathrm{eq}}_{\mathrm{curv\geq0}}(\mathbb{S}^2)$ introduced in \eqref{eq: homeo disc} is a homeomorphism.
\end{Proposition}

\begin{proof}We have already seen in Proposition \ref{prop: realisation RP2} that $\Psi_{\mathbb{S}^2}^{\mathrm{eq}}$ is bijective; therefore, we just need to prove that $\Psi_{\mathbb{S}^2}^{\mathrm{eq}}$ and $\{\Psi_{\mathbb{S}^2}^{\mathrm{eq}}\}^{-1}$ are continuous. However, observe that convergence in the equivariant GH topology implies convergence in the GH topology. Thus, thanks to Theorem \ref{th: homeo S2}, $\{\Psi_{\mathbb{S}^2}^{\mathrm{eq}}\}^{-1}$ is continuous.\\
Let $D_n\to D_{\infty}$ in $\tK$ and let us prove that $\mathcal{D}^{\mathrm{eq}}(\Phi_{\mathbb{S}^2}(D_n),\Phi_{\mathbb{S}^2}(D_{\infty}))\to0$ (see Remark \ref{remark: equivariant GH topology} for the defintion of $\mathcal{D}^{\mathrm{eq}}$), which proves $\Psi_{\mathbb{S}^2}^{\mathrm{eq}}$'s continuity. Since the dimension of compact convex sets is lower semi-continuous with respect to the Hausdorff distance, we can assume that for every $n\in\N$, we have $\dim(D_{\infty})\leq\dim(D_n)$ (forgetting the first terms of the sequence if necessary). Thanks to Lemma \ref{lem: orthogonal}, there exists $\phi_n\to\id_{\R^3}$ in $\Or_3(\R)$ such that $\phi_n^{-1}(\Span(D_{\infty}))\subset\Span(D_n)$. Let us introduce $H_n\coloneqq \phi_n^{-1}(\Span(D_{\infty}))$, $p_n\coloneqq p_{H_n}$ (see Notation \ref{not: orthogonal projection}), $D_n'\coloneqq p_n(D_n)$, and $\epsilon_n\coloneqq\sup_{x\in D_n}\{\de(x,p_n(x))\}$. Observe that $\epsilon_n\to0$. Indeed, given $x\in D_n$, we have $\de(x,p_n(x))\leq\de(x,p_{\infty}(x))+\de(p_n(x),p_{\infty}(x))$, where $p_{\infty}=p_{\Span(D_{\infty})}$. In particular, we have $\de(x,p_n(x))\leq\lvert p_n-p_{\infty}\rvert\Diam(D_n)+\dist_{\mathrm{H}}(D_n,D_{\infty})$. However, since $\phi_n\to\id_{\R^3}$, it is readily checked that $\lvert p_n-p_{\infty}\rvert\to0$; hence, since $\{\Diam(D_n)\}$ is bounded, we have $\epsilon_n\to0$ (note that the proof also works if $\dim(D_{\infty})=1$).\\
Now, let us prove that $\mathcal{D}^{\mathrm{eq}}(\Phi_{\mathbb{S}^2}(D'_n),\Phi_{\mathbb{S}^2}(D_{\infty}))\to0$. First of all, note that $\Phi_{\mathbb{S}^2}(D'_n)$ is equivariantly isometric to $\Phi_{\mathbb{S}^2}(D''_n)$, where $D''_n\coloneqq\phi_n(D_n')$; therefore, we only have to show $\mathcal{D}^{\mathrm{eq}}(\Phi_{\mathbb{S}^2}(D''_n),\Phi_{\mathbb{S}^2}(D_{\infty}))\to0$. Note that $\dist_{\mathrm{H}}(D_n',D_n'')\leq\Diam(D_n)\lvert \phi_n-\id\rvert$, $\dist_{\mathrm{H}}(D_n,D_n')\leq\epsilon_n$; hence, applying the triangle inequality, we have $\dist_{\mathrm{H}}(D_n'',D_{\infty})\leq\dist_{\mathrm{H}}(D_n,D_{\infty})+\Diam(D_n)\lvert \phi_n-\id\rvert+\epsilon_n\to0$. Moreover, observe that $D_n''\subset \Span(D_{\infty})$ (thanks to the properties of $\phi_n$). Thus, there exists $\mu_n\to0$ such that $(1-\mu_n)D_{\infty}\subset D_n''\subset (1+\mu_n)D_{\infty}$ and $(1-\mu_n)D_n''\subset D_{\infty}\subset (1+\mu_n)D_n''$. In particular, applying Lemma \ref{lem: 3 to 3} or \ref{lem: 2 to 2} (depending on $\dim(D_{\infty})$ being $2$ or $3$), we obtain $\mathcal{D}^{\mathrm{eq}}(\Phi_{\mathbb{S}^2}(D''_n),\Phi_{\mathbb{S}^2}(D_{\infty}))\to0$.\\
Observe now that if $\dim(D_{\infty})=\dim(D_n)$ then $D_n=D_n'$ by assumption on $\phi_n$. We will therefore assume that $\dim(D_n')=\dim(D_{\infty})<\dim(D_n)$ to avoid trivialities. Thanks to Lemma \ref{lem: 3 to 2}, and using $\epsilon_n\to0$, we obtain $\mathcal{D}^{\mathrm{eq}}(\Phi_{\mathbb{S}^2}(D_n),\Phi_{\mathbb{S}^2}(D_n'))\to0$.\\
We have now shown that $\mathcal{D}^{\mathrm{eq}}(\Phi_{\mathbb{S}^2}(D_n),\Phi_{\mathbb{S}^2}(D_n'))\to0$ and $\mathcal{D}^{\mathrm{eq}}(\Phi_{\mathbb{S}^2}(D'_n),\Phi_{\mathbb{S}^2}(D_{\infty}))\to0$, which implies $\mathcal{D}^{\mathrm{eq}}(\Phi_{\mathbb{S}^2}(D_n),\Phi_{\mathbb{S}^2}(D_{\infty}))\to0$ thanks to the modified triangle inequality satisfies by $\mathcal{D}^{\mathrm{eq}}$ (see \eqref{eq: modified triangle inequality}).
\end{proof}

\begin{Proposition}\label{prop: disc continuity}The map $\Psi_{\mathbb{D}}\colon\mathscr{K}_{2\leq3}/\Or_3(\R)\to\mathscr{M}_{\mathrm{curv}\geq0}(\mathbb{D})$ introduced in \eqref{eq: homeo disc} is a homeomorphism.
\end{Proposition}

\begin{proof}Let us recall that $\Psi_{\mathbb{D}}$ is a $1$-$1$ correspondence thanks to Proposition \ref{prop: realisation disc}. We will start by proving the continuity of $\Psi_{\mathbb{D}}$. Let us assume that $(D_n,\alpha_n)\to(D_{\infty},\alpha_{\infty})$ in $\mathscr{K}$ and let us prove that $\{\Phi_{\mathbb{D}}(D_n,\alpha_n)\}$ converges to $\Phi_{\mathbb{D}}(D_{\infty},\alpha_{\infty})$ in the GH topology (note that this is stronger than proving $\Psi_{\mathbb{D}}$'s continuity since here we do not ask $\dim(D_n)\in\{2,3\}$). Without loss of generality, we will assume that for every $n\in\N$, we have $\dim(D_{\infty})\leq\dim(D_n)$. Note that if $\dim(D_{\infty})=0$, then, thanks to Lemma \ref{lem: collapsing case}, we have $\Diam(\Phi_{\mathbb{D}}(D_n,\alpha_n))\leq\Diam(\Phi_{\mathbb{S}^2}(D_n))\leq\pi\Diam(D_n)\to0$. In particular, $\Phi_{\mathbb{D}}(D_n,\alpha_n)\to0=\Phi_{\mathbb{D}}(D_{\infty},\alpha_{\infty})$. Now we assume that $1\leq\dim(D_{\infty})$. Let $\phi_n\to\id$, $p_n\to p_{\infty}$, $D_n'$, $D_n''$, and $\epsilon_n\to0$ be defined exactly as in the proof of Proposition \ref{prop: RP2 continuity}, asking that $\phi_n(\alpha_n)=\alpha_{\infty}$ (which is possible thanks to Lemma \ref{lem: orthogonal}).\\
We first prove that $\GH(\Phi_{\mathbb{D}}(D_n',\alpha_n),\Phi_{\mathbb{D}}(D_{\infty},\alpha_{\infty}))\to0$. Observe that $\Phi_{\mathbb{D}}(D_n',\alpha_n)$ is isometric to $\Phi_{\mathbb{D}}(D_n'',\alpha_{\infty})$. Moreover, proceeding exactly as in the proof of Proposition \ref{prop: RP2 continuity}, we can show that there exists $\mu_n\to0$ such that $(1-\mu_n)D_{\infty}\subset D_n''\subset (1+\mu_n)D_{\infty}$ and $(1-\mu_n)D_n''\subset D_{\infty}\subset (1+\mu_n)D_n''$. Let $f_n\colon\Phi_{\mathbb{S}^2}(D_n'')\to\Phi_{\mathbb{S}^2}(D_{\infty})$ and $g_n\colon\Phi_{\mathbb{S}^2}(D_{\infty})\to\Phi_{\mathbb{S}^2}(D_n'')$ be defined as in Lemma \ref{lem: 3 to 3}, \ref{lem: 2 to 2}, or \ref{lem: 1 to 1} (depending on $\dim(D_{\infty})$ being $3$, $2$, or $1$). Observe that in every case, we have $f_n(\Phi_{\mathbb{D}}(D_n'',\alpha_{\infty}))\subset\Phi_{\mathbb{D}}(D_{\infty},\alpha_{\infty})$ and $g_n(\Phi_{\mathbb{D}}(D_{\infty},\alpha_{\infty}))\subset\Phi_{\mathbb{D}}(D_n'',\alpha_{\infty})$; in particular, $\GH(\Phi_{\mathbb{D}}(D_{\infty},\alpha_{\infty})),\Phi_{\mathbb{D}}(D_n'',\alpha_{\infty}))\leq\GH(\Phi_{\mathbb{S}^2}(D_{\infty}),\Phi_{\mathbb{S}^2}(D_n''))\to0$ (using the estimates of the lemmas mentioned before depending on the dimension of $\dim(D_{\infty})$).\\
Now, let us prove that $\GH(\Phi_{\mathbb{D}}(D_n,\alpha_n),\Phi_{\mathbb{D}}(D_n',\alpha_n))\to0$. Observe that if $\dim(D_n)=\dim(D_{\infty})$, then $\phi_n^{-1}(\Span(D_{\infty}))\subset\Span(D_n)$ implies that $D_n=D_n'$. Hence, we will assume that $\dim(D_{\infty})=\dim(D_n')<\dim(D_n)$ to avoid trivialities. Let $f_n\colon\Phi_{\mathbb{S}^2}(D_n)\to\Phi_{\mathbb{S}^2}(D_n')$ be defined as in Lemma \ref{lem: 3 to 2}, \ref{lem: 3 to 1}, or \ref{lem: 2 to 1} (depending on $(\dim(D_n),\dim(D_{\infty}))$ being equal to $(3,2)$, $(3,1)$, or $(2,1)$). It is readily checked that in every case we have $f_n(\Phi_{\mathbb{D}}(D_n,\alpha_n))=\Phi_{\mathbb{D}}(D_n',\alpha_n)$; therefore, thanks to the previous lemmas' estimates, we obtain $\GH(\Phi_{\mathbb{D}}(D_n,\alpha_n),\Phi_{\mathbb{D}}(D_n',\alpha_n))\to0$. This concludes the proof of $\Psi_{\mathbb{D}}$'s continuity.\\

Now, we are going to prove that $\Psi_{\mathbb{D}}^{-1}$ is continuous. Let us assume that $[\mathbb{D},\dist_n]\to[\mathbb{D},\dist_{\infty}]$ w.r.t. the GH topology in $\mathscr{M}_{\mathrm{curv}\geq0}(\mathbb{D})$. Thanks to Proposition \ref{prop: realisation disc}, for every $n\in\N\cup\{\infty\}$ there exists $(D_n,\alpha_n)\in\mathscr{K}_{2\leq3}$ such that $(\mathbb{D},\dist_n)$ is isometric to $\Phi(D_n,\alpha_n)$. We need to show that $[D_n,\alpha_n]\to[D_{\infty},\alpha_{\infty}]$ in $\mathscr{K}_{2\leq3}/\Or_3(\R)$. Note that, since $\mathscr{K}_{2\leq3}/\Or_3(\R)$ is a metric space, it is sufficient to prove that every subsequence of $\{[D_n,\alpha_n]\}$ admits a subsequence converging to $[D_{\infty},\alpha_{\infty}]$. Reindexing the sequence if necessary, let us just prove that $\{[D_n,\alpha_n]\}$ admits a subsequence converging to $[D_{\infty},\alpha_{\infty}]$.\\
First of all, observe that, since $\mathbb{S}^2$ is compact, we can assume that $\alpha_n\to\alpha\in\mathbb{S}^2$ (reindexing the sequence if necessary). Observe that if $\dim(D_n)=3$, then $\Diam(D_n)\leq\Diam(\partial D_n)\leq2\Diam(\partial D_n\cap H_{\alpha_n}^+)=2\Diam(\mathbb{D},\dist_n)$. Moreover, if $\dim(D_n)=2$ and $\alpha_n\perp\Span(D_n)$ then $\Diam(D_n)=\Diam(\mathbb{D},\dist_n)$, if $\alpha_n\in\Span(D_n)$ then $\Diam(D_n)\leq\Diam(\mathcal{D}D_n)\leq2\Diam(\{\mathcal{D}D_n\}_{\alpha_n}^+)=2\Diam(\mathbb{D},\dist_n)$.
Since $\{\Diam(\mathbb{D},\dist_n)\}_n$ is bounded, we can conclude that there exists $r\in(0,\infty)$ such that, for every $n\in\N$, we have $\Diam(D_n)\leq r$. Observe that, for every $n\in\N$, we have $0=s(D_n)\in D_n$. Hence $\{D_n\}$ is a sequence of convex compact subsets of $\overline{B}_r(0)$. Thanks to the Blaschke Theorem (see Theorem 7.3.8 of \cite{Burago-Ivanov_01}), we can assume (passing to a subsequence if necessary) that $D_n\to D$ w.r.t. $\dha$, where $D$ is a compact convex subset of $\overline{B}_r(0)$. Note that since $\alpha_n\to\alpha$, it is readily checked that $D_n^{\alpha_n}\to D^{\alpha}$ (see Notation \ref{not: symmetrisation}). However, $D_n^{\alpha_n}=D_n\to D$; therefore $(D,\alpha)\in\mathscr{K}$. In order to conclude, we just need to prove that there exists $\phi\in\Or_3(\R)$ such that $\phi(D)=D_{\infty}$ and $\phi(\alpha)=\alpha_{\infty}$.\\
Thanks to the first part of the proof, $(D_n,\alpha_n)\to(D,\alpha)$ implies that $\Phi_{\mathbb{D}}(D_n,\alpha_n)\to\Phi_{\mathbb{D}}(D,\alpha)$. However, $\Phi_{\mathbb{D}}(D_n,\alpha_n)\to\Phi_{\mathbb{D}}(D_{\infty},\alpha_{\infty})$ by assumption. Hence, $\Phi_{\mathbb{D}}(D_{\infty},\alpha_{\infty})$ is isometric to $\Phi_{\mathbb{D}}(D,\alpha)$. However, $\Phi_{\mathbb{D}}(D_{\infty},\alpha_{\infty})$ is homeomorphic to $\mathbb{D}$; therefore, we necessarily have $2\leq\dim(D)$, i.e. $(D,\alpha)\in\mathscr{K}_{2\leq3}$. Thus, since $\Psi_{\mathbb{D}}\colon\mathscr{K}_{2\leq3}/\Or_3(\R)\to\mathscr{M}_{\mathrm{curv}\geq0}(\mathbb{D})$ is a $1$-$1$ correspondence, we have $[D,\alpha]=[D_{\infty},\alpha_{\infty}]$, which concludes the proof.\end{proof}

\subsubsection{Moduli spaces of $\RCD(0,2)$-structures}

We finally show that the moduli spaces of $\RCD(0,2)$-structures $\M_{0,2}(\mathbb{S}^2)$, $\M_{0,2}(\R\mathbb{P}^2)$, and $\M_{0,2}(\mathbb{D})$ are contractible.

\begin{Proposition}\label{prop: contractibility S2}The moduli space $\M_{0,2}(\mathbb{S}^2)$ of $\RCD(0,2)$-structures on $\mathbb{S}^2$ is homeomorphic to $\R\times\{\Ks/\Or_3(\R)\}$. In particular, $\M_{0,2}(\mathbb{S}^2)$ is contractible.
\end{Proposition}

\begin{proof}Thanks to Lemma \ref{lem: case of surfaces} and Theorem \ref{th: homeo S2}, $\M_{0,2}(\mathbb{S}^2)$ is homeomorphic to $\R_{>0}\times\{\Ks/\Or_3(\R)\}$ which is itself homeomorphic to $\R\times\{\Ks/\Or_3(\R)\}$. To conclude, we only need to show that $\Ks/\Or_3(\R)$ is contractible. To do so, we define $H\colon I\times\Ks\to\Ks$ by $H(t,D)\coloneqq t\mathbb{B}+(1-t)D$ for every $(t,D)\in I\times\Ks$ (where $+$ is the Minkowski sum and $\mathbb{B}$ is the unit ball in $\R^3$). Observe that $H$ is $\Or_3(\R)$-equivariant, satisfies $H(0,\cdot)=\id_{\Ks}$, and $H(1,\cdot)$ is the constant function equal to $\mathbb{B}$. Hence, to conclude we just need to show that $H$ is continuous. Let $D_1,D_2\in\Ks$ and let $t,s\in I$. Assume that $y\in\mathbb{B}$ and $z\in D_1$, and let $x_t\coloneqq ty+(1-t)z\in H(t,D_1)$ and $x_s\coloneqq sy +(1-s)z\in H(s,D_1)$. Observe that $\de(x_t,x_s)=\lvert t-s\rvert \de(y,z)\leq\lvert t-s\rvert(1+\Diam(D_1))$. Hence, we have:
\begin{equation}\label{eq: continuity of H (I)}
\dha(H(t,D_1),H(s,D_1))\leq\lvert t-s\rvert(1+\Diam(D_1)).
\end{equation}
Observe that there exists $z'\in D_2$ such that $\de(z,z')\leq\epsilon$ (where $\epsilon>0$ is any positive number such that $\dha(D_1,D_2)<\epsilon$). Denoting $x'_s\coloneqq sy+(1-s)z'\in H(s,D_2)$, we have $\de(x_s,x'_s)=(1-s)\de(z,z')\leq(1-s)\epsilon$. Therefore, we have $\dha(H(D_1,s),H(D_2,s))\leq(1-s)\epsilon$ and, letting $\epsilon$ go to $\dha(D_1,D_2)$, we obtain:
\begin{equation}\label{eq: continuity of H (II)}
\dha(H(s,D_1),H(s,D_2))\leq(1-s)\dha(D_1,D_2).
\end{equation}
Applying the triangle inequality to $\dha$ and using \eqref{eq: continuity of H (I)} and \eqref{eq: continuity of H (II)} we finally get:
$$
\dha(H(t,D_1),H(s,D_2))\leq\lvert t-s\rvert(1+\Diam(D_1))+(1-s)\dha(D_1,D_2).
$$
In particular, $H$ is continuous.
\end{proof}



\begin{Proposition}\label{prop: projective plane}The moduli space $\M_{0,2}(\R\mathbb{P}^2)$ of $\RCD(0,2)$-structures on $\R\mathbb{P}^2$ is homeomorphic to $\R\times \{\tK/\Or_3(\R)\}$; in particular, it is contractible.
\end{Proposition}

\begin{proof}First of all, note that thanks to Theorem \ref{th: equiv homeo}, the lift map $p^*\colon\M_{0,2}(\R\mathbb{P}^2)\to\M_{0,2}^{\mathrm{eq}}(\mathbb{S}^{2})$ is a homeomorphism (where $p\colon \mathbb{S}^2\to\mathbb{S}^2/\{\pm1\}=\R\mathbb{P}^2$ is the quotient map). Moreover, applying Lemma \ref{lem: case of surfaces}, $\M_{0,2}^{\mathrm{eq}}(\mathbb{S}^{2})$ is homeomorphic to $\R_{>0}\times\mathscr{M}^{\mathrm{eq}}_{\mathrm{curv\geq0}}(\mathbb{S}^2)$. Thus, thanks to Proposition \ref{prop: RP2 continuity}, $\M_{0,2}^{\mathrm{eq}}(\mathbb{S}^{2})$ is homeomorphic to $\R\times \{\tK/\Or_3(\R)\}$. Now, observe that the map $H\colon I\times\Ks\to\Ks$ introduced in the proof of Proposition \ref{prop: contractibility S2} satisfies $H(I\times\tK)\subset\tK$. Hence, $\tK/\Or_3(\R)$ is contractible and, a fortiori, $\M_{0,2}(\R\mathbb{P}^2)$ is contractible.
\end{proof}

\begin{Proposition}\label{prop: closed disc}The moduli space $\M_{0,2}(\mathbb{D})$ of $\RCD(0,2)$-structures on $\mathbb{D}$ is homeomorphic to $\R\times\{\mathscr{K}_{2\leq3}/\Or_3(\R)\}$ (where $\mathscr{K}_{2\leq3}$ is introduced in Notation \ref{not: moduli space of the closed disc}); in particular, it is contractible.
\end{Proposition}

\begin{proof}Observe that thanks to Lemma \ref{lem: case of surfaces}, $\M_{0,2}(\mathbb{D})$ is homeomorphic to $\R_{>0}\times\mathscr{M}_{\mathrm{curv\geq0}}(\mathbb{D})$ which is itself homeomorphic to $\R\times\mathscr{M}_{\mathrm{curv\geq0}}(\mathbb{D})$. Thus, thanks to Proposition \ref{prop: disc continuity}, $\M_{0,2}(\mathbb{D})$ is homeomorphic to $\R\times\{\mathscr{K}_{2\leq3}/\Or_3(\R)\}$. Now let us consider the map $H_{\mathbb{D}}\colon I\times\mathscr{K}_{2\leq3}\to\mathscr{K}_{2\leq3}$ defined by $H_{\mathbb{D}}(t,D,\alpha)\coloneqq(H(t,D),\alpha)$, where $H\colon I\times\Ks\to\Ks$ is introduced in the proof of Proposition \ref{prop: contractibility S2}. Observe that $H_{\mathbb{D}}$ is continuous since $H$ is continuous. Moreover, note that $H_{\mathbb{D}}$ is equivariant w.r.t. the action of $\Or_3(\R)$ on $\mathscr{K}_{2\leq3}$. To conclude, note that since $\Or_3(\R)$ acts transitively on $\mathbb{S}^2$, we have $[H_{\mathbb{D}}(1,\cdot)]\equiv[\mathbb{B},(0,0,1)]\in\mathscr{K}_{2\leq3}/\Or_3(\R)$; hence, $\mathscr{K}_{2\leq3}/\Or_3(\R)$ is contractible.
\end{proof}

\newpage

\bibliographystyle{abbrv}
\bibliography{biblio}

\end{document}